\documentclass{article}
\usepackage{blindtext}
\usepackage{geometry}
\usepackage{amsmath}
\usepackage[normalem]{ulem}
\usepackage{amsfonts}
\usepackage{amssymb}
\usepackage{amsth m}
\usepackage[utf8]{inputenc}
\usepackage{enumerate}
\usepackage{changepage}
\usepackage[square, numbers]{natbib}  % For square bracketed references
\usepackage{graphicx}
\usepackage{subcaption}
\usepackage{float}
\usepackage{enumitem}
\usepackage{xcolor}
\usepackage{easyReview}
\usepackage{lineno}
\usepackage{hyperref}
\usepackage[capitalise]{cleveref}

\theoremstyle{remark}
\newtheorem{remark}{Remark}

\usepackage{setspace}
\newtheoremstyle{custom} % Define a custom style
  {3pt} % Space above
  {3pt} % Space below
  {\itshape} % Body font (use \normalfont to make it non-italic)
  {} % Indent amount
  {\bfseries} % Theorem head font
  {.} % Punctuation after theorem head
  {.5em} % Space after theorem head
  {} % Theorem head spec

\theoremstyle{custom} % Apply the custom style

\newtheorem{theorem}{Theorem}
\newtheorem{proposition}{Proposition}
\begin{document}

%\maketitle

%\vspace{10pt}
%\setlength{\marginparwidth}{0pt}
%\setlength{\marginparsep}{0pt}
%\marginparwidth = 10pt

\begin{center}

{\Large \bf Oscillatory Regimes in a Game-Theoretic Model for Mosquito Population Dynamics under Breeding Site Control} 

\vspace{.25cm}
 
{\large Mohammad Rubayet Rahman}\footnote[1]{Department of Mathematics, Oklahoma State University, Stillwater, OK 74078, e-mail: rubayet.rahman$@$okstate.edu},
{\large Chanaka Kottegoda}\footnote[2]{Department of Mathematics and Physics, Marshall University, Huntington, WV, e-mail: kottegoda@marshall.edu },
{\large and Lucas M. Stolerman}\footnote[3]{Department of Mathematics, Oklahoma State University, Stillwater, OK 74078, e-mail: lucas.martins$\_$stolerman$@$okstate.edu}
\vspace{.03cm}
\end{center}

\date{}

\begin{abstract}

Mosquito-borne diseases remain a major public-health threat, and the effective control of mosquito populations requires sustained household participation in removing breeding sites. While environmental drivers of mosquito oscillations have been extensively studied, the influence of spontaneous household decision-making on the dynamics of mosquito populations remains poorly understood. We introduce a game-theoretic model in which the fraction of households performing breeding site control evolves through imitation dynamics driven by perceived risks. Household behavior regulates the carrying capacity of the aquatic mosquito stage, creating a feedback between control actions and mosquito population growth. For a simplified model with constant payoffs, we characterize four locally stable equilibria, corresponding to full or no household control and the presence or absence of mosquito populations. When the perceived risk of not controlling breeding sites depends on mosquito prevalence, the system admits an additional equilibrium with partial household engagement. We derive conditions under which this equilibrium undergoes a Hopf bifurcation, yielding sustained oscillations arising solely from the interaction between mosquito abundance and household behavior. Numerical simulations and parameter explorations further describe the amplitude and phase properties of these oscillatory regimes.

\end{abstract}
%\tableofcontents

\parskip=6pt

\section{Introduction}

Mosquito-borne diseases (MBDs) such as dengue, Zika, chikungunya and yellow fever continue to pose significant public-health challenges worldwide \cite{CDC_deadliest_animal}. For example, recent estimates indicate that 100 to 400 million dengue infections occur each year, resulting in a substantial burden of severe cases and deaths \cite{who2024dengue}. MBD transmission occurs through bites from infected female mosquitoes, which acquire the virus when feeding on infected humans and subsequently transmit it to susceptible individuals. In the absence of universally effective vaccines or large-scale therapeutic options for most MBDs, public-health strategies rely heavily on reducing mosquito populations \cite{wilson2020importance}. A particular challenge involves mosquito preference for biting and feeding on humans. Species such as the \emph{Ae. aegypti} and \emph{Ae. albopictus}  are highly anthropophilic and primarily breed in small, artificial containers near households. Hence, effective control of mosquito populations must include the removal of breeding sites \cite{gubler2011dengue,wong2011oviposition,buhler2019environmental}.

Community involvement is critical for reducing or eliminating larval mosquito habitats in urban areas \cite{WHO_engaging_communities_dengue}. Depending on the effectiveness of communication from local governments and on the perception of the danger posed by MBDs, households may be more or less inclined to engage in breeding site control \cite{raude2012public,dussault2022arbovirus}. Strong engagement can be driven by the perceived advantages of controlling mosquito populations, despite the labor required to remove water containers.

Mosquito population dynamics can be complex and driven by an interplay of environmental and biological factors. Empirical studies have reported strong seasonal components and oscillatory patterns associated with fluctuations in temperature and rainfall \cite{whittaker2022novel,lana2018assessment,shaman2007reproductive}. Mathematical models have successfully elucidated mechanisms underlying these oscillations, including the influence of climate-driven parameters, nonlinearities, and critical developmental delays \cite{abdelrazec2017mathematical,okuneye2018mathematical, wan2014new,juan2025two}. However, the role of human behavior in shaping mosquito dynamics remains poorly understood despite mounting evidence of its relevance \cite{renard2023interaction,thongsripong2021human}. In the context of breeding site control, existing models have examined oscillations arising from fluctuations in environmental consciousness \cite{predescu2006analysis,predescu2007dynamics} or from impulsive mechanical removal driven by periodic human interventions \cite{dumont_mosquito_control}. Yet, to the best of our knowledge, no modeling work has investigated how spontaneous household engagement, modulated by risk perception, may generate nontrivial dynamics.

Here we introduce a model to investigate mosquito population dynamics under breeding site control. Inspired by a vaccination game with imitation dynamics proposed by Bauch \cite{bauch2005imitation}, we adopt a game-theoretic framework in which the fraction of households performing control is determined by payoffs based on perceived risks.  To integrate the behavioral and entomological components, we assume a behavior-dependent carrying capacity for the aquatic population. For a simpler model with constant payoffs, we identify four locally stable steady states corresponding to full or no breeding site control and the presence or absence of mosquito populations. We then extend the model by making the payoff for not controlling breeding sites dependent on mosquito prevalence, capturing the idea that households perceive a higher risk when they frequently encounter mosquitoes. In this case, we obtain a fifth steady state with partial breeding site control, and the model undergoes a Hopf bifurcation under certain conditions. We complement our analytical results with numerical simulations and a parametric exploration of the amplitude and phase characteristics of the resulting stable oscillations.

The paper is organized as follows. In Section \ref{model}, we present the behavioral–entomological model, detailing the mosquito population dynamics,  the imitation dynamics governing household control,  and the integration with a behavior-dependent carrying capacity. Section \ref{analysis} presents the analytical results for the model with constant payoffs and the full game-theoretic model with a prevalence-dependent payoff, including the analysis leading to the identification of a Hopf bifurcation. Section \ref{numerics} complements the theoretical findings with numerical simulations illustrating the trajectories, oscillatory regimes, and the dependence of amplitude and phase on key parameters. Finally,  in Section \ref{conclusions} we discuss the implications of our results, limitations of the modeling framework, and potential avenues for future research.

\section{Model formulation} 
\label{model}

\subsection{Mosquito population dynamics}

The mosquito life cycle consists of four stages: egg, larva, pupa, and adult. Here we group the first three stages into a single aquatic population compartment. The adult population compartment represents fully developed mosquitoes capable of flight, reproduction, and disease transmission.  Following the approach proposed by  Dumont and Thuilliez  \cite{dumont_mosquito_control},  we adopt the entomological model:
\begin{equation}
\begin{aligned}
   \frac{dL_v}{dt} &  = r b A_v \left(1 - \frac{L_v}{K_v}\right) - (\nu_L + \mu_L) L_v \\
     \frac{dA_v}{dt} & = \nu_L L_v - \mu_A A_v.
\end{aligned}
  \label{Eq:entomological_model}
\end{equation}

Here, $L_v$  and $A_v$ represent the aquatic and adult population compartments, respectively. The sex ratio $r$ represents the proportion of females in the overall mosquito population. The egg-laying rate $b$ denotes the average number of eggs laid per unit of time by a female mosquito. The transition rate $\nu_L$ captures the rate at which mosquitoes progress from the aquatic stage to adulthood.  The parameters  $\mu_L$ and $\mu_A$ define the natural death rates for mosquitoes in their aquatic and adult stages, respectively. Finally, $K_v$ is the carrying capacity, representing the availability of mosquito breeding sites. The non-linear term \(rbA_v\left(1 - \frac{L_v}{K_v} \right) \) represents the growth rate of aquatic mosquitoes. The basic offspring number, defined as
$$N = \frac{\nu_L r b}{(\nu_L + \mu_L)\mu_A}$$
quantifies the average number of adult mosquitoes produced by a single adult throughout its lifetime.

\subsection{Household control of mosquito breeding sites}

 We assume that households decide whether to engage in mechanical control of mosquito breeding sites, hereafter referred to as breeding site control or just control.  Following the imitation game framework proposed in the context of vaccination strategies \cite{bauch2005imitation,wang2016statistical}, we define the perceived payoff for performing breeding site control as $f_c=-r_c$, where $r_c>0$ is the perceived risk of experiencing a burden, such as time, or monetary costs associated with the control effort. The perceived payoff for households not performing control, $f_d$, is considered to depend on the perceived risk of mosquito-borne infection $(r_d>0)$ and the perceived risk exposure to infectious mosquito bites over time, which is assumed to increase with the adult mosquito population $A_v$. For simplicity, here we adopt a linear model

$$
f_d\left(A_v\right)=-r_d  m  A_v
$$

where $m$ is a parameter quantifying the sensitivity of households' behavior to changes in mosquito prevalence. The imitation dynamics assume that households tend to adopt strategies with higher perceived payoffs \cite{chang2020game}. Hence, the fractions $w$ and $\bar{w}$ of households performing and not performing breeding site control satisfy the system:
\begin{equation*}
\textcolor{black}{
\begin{aligned}
 \displaystyle  \frac{dw}{dt} &  = k f_c \bar{w} w - k f_d \bar{w} w    \\
 \displaystyle  \frac{d\bar{w}}{dt} &  =  k f_d \bar{w} w - k f_c \bar{w} w  
\end{aligned}
}
\end{equation*}

where $k$ is a scale coefficient representing the imitation rate. Since $\bar{w} = 1 - w$, a differential equation governing $w$ is given by
 \begin{equation}
\begin{split}
\frac{d w}{d t}& =k w(1-w)\left(f_c-f_d\left(A_v\right)\right)\\
 & =  k  w(1-w)\left(-r_c+r_d  m  A_v\right).
\end{split}
\label{Eq:imitation_game}
\end{equation}
 This model structure mirrors the vaccination game dynamics proposed by Bauch in \cite{bauch2005imitation}, with breeding site control and adult mosquito populations playing the role of vaccination and disease prevalence, respectively.

\subsection{Model integration with a behavior-dependent carrying capacity} 

Our model captures the impact of mosquito control measures on available breeding sites, represented by the carrying capacity $K_v$. Figure~\ref{fig:schematics} illustrates the integration of the entomological model~\eqref{Eq:entomological_model} with the behavior dynamics given by Equation~\eqref{Eq:imitation_game}. To incorporate the effect of household control on mosquito populations, we assume that the growth rate $\Lambda$ depends on $w$, which regulates the carrying capacity of the aquatic stage. Specifically, we define 
\(\Lambda(w) = rbA_v \left(1 - \frac{L_v}{K_v(w)} \right)\), 
where \(K_v(w)\) is a linear function that attains its maximum value \(K_{v_{\text{max}}}\) when \(w = 0\) and minimum value \(K_{v_{\text{min}}}\) when \(w = 1\).

\begin{figure}[H]  
  \centering
  \includegraphics[width=0.6\textwidth]{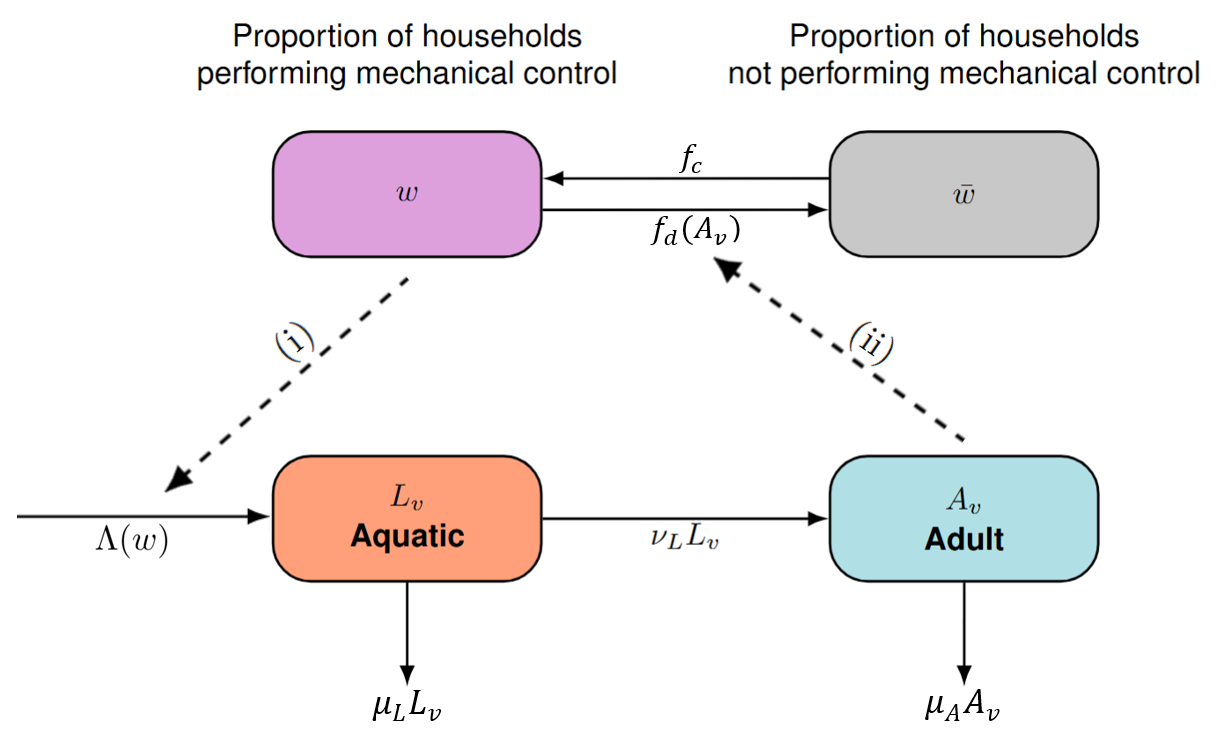}  
\caption{
We model the dynamic impact of household mosquito control measures on breeding site availability. Here, \( w \) and \( \bar{w} \) denote the proportions of households performing and not performing breeding site control, respectively. Dashed arrows indicate how the entomological and behavioral components are integrated: (i) the growth rate of aquatic mosquitoes, \( \Lambda(w) \), as a function of the proportion of households performing control; and (ii) the perceived payoff for not performing control, \( f_d(A_v) \), as a function of the adult mosquito population \( A_v \).
}
 \label{fig:schematics}   
\end{figure}

The model equations are given by the following system: 
\begin{equation}
\begin{aligned}
\frac{dL_v}{dt} &= rbA_v \left(1 - \frac{L_v}{K_v(w)} \right) - (\nu_L + \mu_L)L_v, \\
\frac{dA_v}{dt} &= \nu_L L_v - \mu_A A_v, \\
\frac{dw}{dt}   &= k w(1 - w)\left[-r_c + r_d m A_v\right].
\end{aligned}
\label{eq:behavioral_entomological_system}
\end{equation}

where $K_v(w) = K_{v_{\text{max}}} - w(K_{v_{\text{max}}} - K_{v_{\text{min}}})$. Intuitively, when the payoff for performing breeding site control exceeds that of not performing it (i.e., \( f_c > f_d  \) or $r_d m A_v > r_c $), the fraction of households engaging in control increases (\( \frac{dw}{dt} > 0 \)), leading to a reduction in breeding sites as reflected by a decrease in the carrying capacity \( K_v(w) \). Conversely, when \( f_c < f_d \), household participation declines (\( w \) decreases), increasing available breeding sites. The initial conditions for the system at time \( t = 0 \) are all non-negative, with \( w(0) \in [0, 1] \). Since the right-hand side of System~\eqref{eq:behavioral_entomological_system} is continuously differentiable, the Cauchy-Lipschitz theorem guarantees that the initial value problem admits a unique maximal solution. Moreover, the set
\[
K \;=\; \Bigl\{(L_v,A_v,w)\in \mathbb{R}^3:\; 0\le L_v\le K_{v_{\max}},\;\; 0\le A_v\le \tfrac{\nu_L}{\mu_A}K_{v_{\max}}, \  0 \leq w \leq 1\Bigr\}.
\]
is positively invariant by~\eqref{eq:behavioral_entomological_system}. In what follows, we present a steady state and local stability analysis of our model. To build intuition, we first examine the simpler case where \( f_d =- r_d \) (section \ref{simple_model}). We refer to System~\eqref{eq:behavioral_entomological_system} as the \emph{full game-theoretic model}, which is analyzed in section \ref{full_GTM}.

\section{Steady states and local stability analysis} \label{analysis}

\subsection{A simplified model with constant payoffs} 
\label{simple_model}

We first consider the case where the perceived payoff for not performing breeding site control is given by \( f_d = -r_d \), with \( r_d \) representing a constant perceived risk of mosquito-borne infection. This case paves the way for the analysis of System~\eqref{eq:behavioral_entomological_system}. It is also motivated by scenarios in which individuals' perception of exposure to mosquito bites remains stable over time—for instance, in environments with limited access to entomological information, or where perceptions are shaped primarily by habit, cultural norms, or long-term awareness campaigns. The model equations are given by the following system: 
\begin{equation} 
\begin{aligned}
\frac{dL_v}{dt} &= rbA_v \left(1 - \frac{L_v}{K_v(w)} \right) - (\nu_L + \mu_L)L_v, \\
\frac{dA_v}{dt} &= \nu_L L_v - \mu_A A_v, \\
\frac{dw}{dt}   &= k w(1 - w)\left[-r_c + r_d\right],
\end{aligned}
\label{eq:simplified_model_1}
\end{equation}

where $K_v(w) = K_{v_{\text{max}}} - w(K_{v_{\text{max}}} - K_{v_{\text{min}}})$.

\textbf{Steady States}

To find biologically plausible (nonnegative) steady states of System~\eqref{eq:simplified_model_1}, we must find $L_v, A_v, w$ such that the time derivatives of all model components are zero. Therefore, we must solve the following system:
\begin{equation*}
\begin{aligned}
rb A_v \left(1 - \frac{L_v}{K_v(w)} \right) - (\nu_L + \mu_L) L_v &= 0, \\
\nu_L L_v - \mu_A A_v &= 0, \\
kw(1 - w)(-r_c + r_d) &= 0.
\end{aligned}
\end{equation*}

We list the steady states in Proposition \ref{prop: Baseline Model}. The proof can be found in the supplementary material. 

\begin{proposition}\label{prop: Baseline Model}
The steady states \((L_v^*, A_v^*, w^*)\) of System~\eqref{eq:simplified_model_1}, along with their interpretation, are as follows:

\underline{\text{$E_{01}$:}} $\left(L_v^*, A_v^*, w^*\right)=(0,0,0)$: Mosquito-free, no breeding site control.

\underline{\text{$E_{02}$:}} $\left(L_v^*, A_v^*, w^*\right)=(0,0,1)$: Mosquito-free, full breeding site control.

\underline{\text{$E_{03}$:}}
$$
\left(L_v^*, A_v^*, w^*\right)=\left(K_{v_{\max }}\left(1-\frac{1}{{N}}\right), \frac{\nu_L}{\mu_A} K_{v_{\max }}\left(1-\frac{1}{{N}}\right), 0\right)
$$
if and only if ${N}>1$: Mosquito-positive, no breeding site control.

\underline{\text{$E_{04}$:}}
$$
\left(L_v^*, A_v^*, w^*\right)=\left(K_{v_{\min }}\left(1-\frac{1}{{N}}\right), \frac{\nu_L}{\mu_A} K_{v_{\min }}\left(1-\frac{1}{{N}}\right), 1\right)
$$
if and only if ${N}>1$: Mosquito-positive, full breeding site control.

\end{proposition}

Having obtained explicit expressions for the steady states and the parameter regimes in which they exist, we analyze their local stability under small perturbations. Using the classical linearization approach based on the Hartman–Grobman Theorem~\cite{wiggins2003introduction,strogatz2018nonlinear}, we derive conditions under which each steady state is locally asymptotically stable (LAS) or unstable. The Jacobian matrix for the system is given by
\begin{equation}
J = \begin{bmatrix}
-\frac{r b A_v}{K_v(w)} - (\nu_L + \mu_L) & r b \left(1 - \frac{L_v}{K_v(w)}\right) & \frac{- r b A_v L_v (K_{v_{\max}} - K_{v_{\min}})}{(K_v(w))^2} \\
\nu_L & -\mu_A & 0 \\
0 & 0 & k(-r_c+r_d) (1 - 2w)
\end{bmatrix}.
\label{jacobian_model_1}
\end{equation}

We evaluate the Jacobian matrix \eqref{jacobian_model_1} at each steady state and compute its eigenvalues. The steady state is LAS if all eigenvalues have negative real parts, and unstable if at least one eigenvalue has a positive real part. The stability conditions are summarized in the following theorem.

\begin{theorem}\label{thm: Baseline Model}
The local stability of the steady states for System~\eqref{eq:simplified_model_1} is characterized as follows:
\begin{enumerate}
\item If \( r_c - r_d > 0 \), and $N<1$ then $E_{01}$ is LAS; $E_{02}$ is unstable.

\item If \( r_c - r_d < 0 \), and $N<1$ then $E_{02}$ is LAS; $E_{01}$ is unstable.

\item If \( r_c - r_d > 0 \), and $N>1$ then $E_{03}$ is LAS; $E_{01}$, $E_{02}$, and $E_{04}$ are unstable.

\item If \( r_c - r_d < 0 \), and $N>1$ then $E_{04}$ is LAS; $E_{01}$, $E_{02}$, and $E_{03}$ are unstable.

\end{enumerate}

\end{theorem}
\color{black}

\begin{proof}

 We proceed to analyze each steady state separately, proving the claims from items 1--4.

 \begin{itemize}

 \item \underline{$E_{01}$:} We evaluate the Jacobian matrix at the steady state $E_{01}$ and find the eigenvalues. Then we derive the conditions for stability. The Jacobian matrix at $E_{01}$ is given by

\[
J = \begin{bmatrix}
-(\nu_L + \mu_L) & r b & 0 \\
\nu_L & -\mu_A & 0 \\
0 & 0 & k(r_d - r_c)
\end{bmatrix}
\]

and the resulting characteristic equation is given by \(\det(J - \lambda I) = 0\), where

\[
J - \lambda I = \begin{bmatrix}
-(\nu_L + \mu_L) - \lambda & r b & 0 \\
\nu_L & -\mu_A - \lambda & 0 \\
0 & 0 & k(r_d - r_c) - \lambda
\end{bmatrix}.
\]

The determinant of this matrix can be calculated by expanding along the third row, i.e.,

$$
\begin{aligned}
\operatorname{det}(J-\lambda I) & =(k(r_d - r_c)-\lambda)  \operatorname{det}\left[\begin{array}{cc}
-\left(\nu_L+\mu_L\right)-\lambda & r b \\
\nu_L & -\mu_A-\lambda
\end{array}\right] \\
& =(k(r_d - r_c)-\lambda)\left(\lambda^2+\left[\left(\nu_L+\mu_L\right)+\mu_A\right] \lambda+\left[\left(\nu_L+\mu_L\right) \mu_A-r b \nu_L\right]\right).
\end{aligned}
$$

A first eigenvalue is thus given by \(\lambda_1= k(r_d - r_c)\), and the other two are roots of the quadratic polynomial, i.e., 
$$\lambda_{2,3}=\frac{-\left[(\nu_L+\mu_L)+\mu_A\right] \pm \sqrt{\left[(\nu_L+\mu_L)+\mu_A\right]^2 - 4\left[(\nu_L+\mu_L)\mu_A - r b \nu_L\right]}}{2}.$$

The eigenvalue $\lambda_1$ will be negative if $k(r_d - r_c) < 0 \Leftrightarrow r_c - r_d >0$ and the other two eigenvalues \(\lambda_{2,3}\) will have negative real parts if $$
\begin{aligned}
& \left(\nu_L+\mu_L\right) \mu_A-r b \nu_L>0 
& \Leftrightarrow\quad {N}<1.
\end{aligned}
$$

Hence, \( E_{01} \) is LAS, which establishes the stability condition stated in item 1. On the other hand, \( E_{01} \) becomes unstable if at least one eigenvalue is positive, which occurs when \( r_d - r_c > 0 \), or when \( N > 1 \).

  \item \underline{$E_{02}$:} The 
characteristic equation is given by \(\det(J - \lambda I) = 0\), where
$$J = \begin{bmatrix}
-(\nu_L + \mu_L) & r b & 0 \\
\nu_L & -\mu_A & 0 \\
0 & 0 & k(r_c - r_d)
\end{bmatrix}.$$

The determinant can be calculated by expanding along the third row, i.e.,

$\begin{aligned} \operatorname{det}(J-\lambda I) & =\left(k\left(r_c-r_d\right)-\lambda\right) \operatorname{det}\left[\begin{array}{cc}-\left(\nu_L+\mu_L\right)-\lambda & r b \\ \nu_L & -\mu_A-\lambda\end{array}\right] \\ & =\left(k\left(r_c-r_d\right)-\lambda\right)\left(\lambda^2+\left[\left(\nu_L+\mu_L\right)+\mu_A\right] \lambda+\left[\left(\nu_L+\mu_L\right) \mu_A-r b \nu_L\right]\right)\end{aligned}$

The first eigenvalue is \(\lambda_1 = k(r_c - r_d)\) and the other two eigenvalues are given by 
$$\lambda_{2,3} = \frac{-\left[(\nu_L + \mu_L) + \mu_A\right] \pm \sqrt{\left[(\nu_L + \mu_L) + \mu_A\right]^2 - 4\left[(\nu_L + \mu_L) \mu_A - r b \nu_L\right]}}{2}.$$

Since $k>0$, the eigenvalue \(\lambda_1\) will be negative if and only if $r_c - r_d<0$. Moreover, since  $(\nu_L + \mu_L) + \mu_A > 0 $,  the other two eigenvalues \(\lambda_{2,3}\) will have negative real parts if and only if
$$
\begin{aligned}
& \left(\nu_L+\mu_L\right) \mu_A-r b \nu_L>0 
& \Leftrightarrow\quad {N}<1.
\end{aligned}
$$

Thus, $E_{02}$ is LAS, as outlined in item 2, and becomes unstable when either $r_c-r_d>0$ or $N>1$.

 \item \underline{$E_{03}$:}  The Jacobian matrix at $E_{03}$ is given by

\[
J = \begin{bmatrix}
-\left(\frac{r b \nu_L}{\mu_A} \left(1 - \frac{1}{{N}}\right) + \nu_L + \mu_L\right) & \frac{r b}{{N}} & -r b \frac{\nu_L}{\mu_A} \left(1 - \frac{1}{{N}}\right)^2 \left(K_{v_{\max}} - K_{v_{\min}}\right) \\
\nu_L & -\mu_A & 0 \\
0 & 0 & k(r_d - r_c)
\end{bmatrix}
\]

and the characteristic equation is  \(\det(J - \lambda I) = 0\), where

\[
J - \lambda I = \begin{bmatrix}
-\left(\frac{r b \nu_L}{\mu_A} \left(1 - \frac{1}{{N}}\right) + \nu_L + \mu_L\right) - \lambda & \frac{r b}{{N}} & -r b \frac{\nu_L}{\mu_A} \left(1 - \frac{1}{{N}}\right)^2 \left(K_{v_{\max}} - K_{v_{\min}}\right) \\
\nu_L & -\mu_A - \lambda & 0 \\
0 & 0 & k(r_d - r_c) - \lambda
\end{bmatrix}.
\]

The determinant of this matrix can be calculated by expanding along the third row, i.e.,

$$
\begin{aligned}
\operatorname{det}(J-\lambda I)&=[k(r_d - r_c)-\lambda] \operatorname{det}\left[\begin{array}{cc}
-\left(\frac{r b \nu_L}{\mu_A}\left(1-\frac{1}{{N}}\right)+\nu_L+\mu_L\right)-\lambda & \frac{r b}{{N}} \\
\nu_L & -\mu_A-\lambda
\end{array}\right] \\
& =[k(r_d - r_c)-\lambda] \\
& {\left[\lambda^2+\left(\left(\frac{r b \nu_L}{\mu_A}\left(1-\frac{1}{{N}}\right)+\nu_L+\mu_L\right)+\mu_A\right) \lambda+\left(\left(\frac{r b \nu_L}{\mu_A}\left(1-\frac{1}{{N}}\right)+\nu_L+\mu_L\right) \mu_A-\frac{r b \nu_L}{{N}}\right)\right]}.
\end{aligned}
$$

The first eigenvalue is given by \(\lambda_1 = k(r_d - r_c)\) and the other two eigenvalues are given by $$
\begin{aligned}
& \lambda_{2,3}=\frac{-\left(\left(\frac{r b \nu_L}{\mu_A}\left(1-\frac{1}{{N}}\right)+\nu_L+\mu_L\right)+\mu_A\right)}{2} \pm \\
& \frac{\sqrt{\left(\left(\frac{r b \nu_L}{\mu_A}\left(1-\frac{1}{{N}}\right)+\nu_L+\mu_L\right)+\mu_A\right)^2-4\left(\left(\frac{r b \nu_L}{\mu_A}\left(1-\frac{1}{{N}}\right)+\nu_L+\mu_L\right) \mu_A-\frac{r b \nu_L}{{N}}\right)}}{2}.
\end{aligned}
$$

The eigenvalue \(\lambda_1\) will be negative if and only if  $r_c - r_d > 0$. The eigenvalues \(\lambda_{2,3}\) will have negative real parts if and only if 

$$
\begin{aligned}
& \left(\frac{r b \nu_L}{\mu_A}\left(1-\frac{1}{{N}}\right)+\nu_L+\mu_L\right) \mu_A-\frac{r b \nu_L}{{N}}>0 \\
& \Leftrightarrow \quad\left(1-\frac{2}{{N}}\right)>-\frac{\left(\nu_L+\mu_L\right) \mu_A}{r b \nu_L} = \frac{1}{{N}} \\
\end{aligned}
$$

which is true if $N>1$ (the existence condition for $E_{03}$). Therefore, $E_{03}$ is LAS, confirming the stability conditions stated in item 3. Conversely, $E_{03}$ is unstable if $r_d-r_c>0$.

  \item \underline{$E_{04}$:} The Jacobian matrix at $E_{04}$ is given by

\[
J = \begin{bmatrix}
-\left(\frac{r b \nu_L}{\mu_A} \left(1 - \frac{1}{{N}}\right) + \nu_L + \mu_L\right) & \frac{r b}{{N}} & -r b \frac{\nu_L}{\mu_A} \left(1 - \frac{1}{{N}}\right)^2 \left(K_{v_{\max}} - K_{v_{\min}}\right) \\
\nu_L & -\mu_A & 0 \\
0 & 0 & k(r_c - r_d)
\end{bmatrix}
\]

and the characteristic equation becomes \(\det(J - \lambda I) = 0 \), where

\[
J - \lambda I = \begin{bmatrix}
-\left(\frac{r b \nu_L}{\mu_A} \left(1 - \frac{1}{{N}}\right) + \nu_L + \mu_L\right) - \lambda & \frac{r b}{{N}} & -r b \frac{\nu_L}{\mu_A} \left(1 - \frac{1}{{N}}\right)^2 \left(K_{v_{\max}} - K_{v_{\min}}\right) \\
\nu_L & -\mu_A - \lambda & 0 \\
0 & 0 & k(r_c - r_d) - \lambda
\end{bmatrix}.
\]

The determinant of this matrix can be calculated by expanding along the third row, i.e.,

$$
\operatorname{det}(J-\lambda I)=[k(r_c - r_d)-\lambda] \operatorname{det}\left[\begin{array}{cc}
-\left(\frac{r b \nu_L}{\mu_A}\left(1-\frac{1}{{N}}\right)+\nu_L+\mu_L\right)-\lambda & \frac{r b}{{N}} \\
\nu_L & -\mu_A-\lambda
\end{array}\right]
$$
 
which yields 
$$
\begin{aligned}
& \operatorname{det}(J-\lambda I) =[k(r_c - r_d)-\lambda]  \ \cdot \  \\ 
& {\left[\lambda^2+\left(\left(\frac{r b \nu_L}{\mu_A}\left(1-\frac{1}{{N}}\right)+\nu_L+\mu_L\right)+\mu_A\right) \lambda+\left(\left(\frac{r b \nu_L}{\mu_A}\left(1-\frac{1}{{N}}\right)+\nu_L+\mu_L\right) \mu_A-\frac{r b \nu_L}{{N}}\right)\right]}.
\end{aligned}
$$

The first eigenvalue is given by \(\lambda_1 = k(r_c - r_d)\) and the other two eigenvalues are given by 

$$\begin{aligned}
& \lambda_{2,3}=\frac{-\left(\left(\frac{r b \nu_L}{\mu_A}\left(1-\frac{1}{{N}}\right)+\nu_L+\mu_L\right)+\mu_A\right)}{2} \pm \\
& \frac{\sqrt{\left(\left(\frac{r b \nu_L}{\mu_A}\left(1-\frac{1}{{N}}\right)+\nu_L+\mu_L\right)+\mu_A\right)^2-4\left(\left(\frac{r b \nu_L}{\mu_A}\left(1-\frac{1}{{N}}\right)+\nu_L+\mu_L\right) \mu_A-\frac{r b \nu_L}{{N}}\right)}}{2}.
\end{aligned}
$$

Now the eigenvalue \(\lambda_1\) will be negative if and only if $r_c - r_d<0$ and the other two eigenvalues \(\lambda_{2,3}\) will have negative real parts if

$$
\begin{aligned}
& \left(\frac{r b \nu_L}{\mu_A}\left(1-\frac{1}{{N}}\right)+\nu_L+\mu_L\right) \mu_A-\frac{r b \nu_L}{{N}}>0
\end{aligned}
$$

which is the same condition obtained for $E_{03}$. Since $N>1$ is also an existence condition for $E_{04}$, the analysis confirms that $E_{04}$ is LAS, establishing the stability conditions stated in item 4. On the other hand, $E_{04}$ loses stability if $r_c-r_d >0$.

\end{itemize}

\end{proof}

 The results from Proposition~\ref{prop: Baseline Model} and Theorem~\ref{thm: Baseline Model} are summarized in Table~\ref{tab:LAS_summary_rigorous}. The basic offspring number $N$ and the difference between perceived risks $r_c - r_d$ delineate four parameter regimes, each characterized by a distinct LAS equilibrium, while the remaining steady states are unstable. These findings are biologically interpretable: $N$ and $r_c - r_d$ are independent quantities, and $N > 1$ is known to be the necessary and sufficient condition for the mosquito-free steady states to lose stability \cite{dumont_mosquito_control}. The dynamics of the proportion of households performing control, $w$, are decoupled from the mosquito population dynamics in this simplified model. Specifically, $w \to 1$ if and only if $r_c - r_d < 0$ (that is, $r_d > r_c$), meaning that the perceived risk of \emph{not} performing breeding site control exceeds that of performing control.

\begin{table}[ht]
\centering
\renewcommand{\arraystretch}{1.7}
\begin{tabular}{|c|c|c|}
\hline
\textbf{Parameter Regimes} & \boldmath{$N < 1$} & \boldmath{$N > 1$} \\
\hline
\boldmath{$r_c - r_d > 0$} 
& \begin{tabular}[c]{@{}l@{}} 
 $E_{01}$ is LAS. \\
 $E_{02}$ is unstable. \\
 $E_{03}$ and $E_{04}$ do not exist.
\end{tabular}
& \begin{tabular}[c]{@{}l@{}} 
 $E_{03}$ is LAS. \\
 $E_{01}, E_{02}, E_{04}$ are unstable.
\end{tabular} \\
\hline
\boldmath{$r_c - r_d < 0$} 
& \begin{tabular}[c]{@{}l@{}} 
$E_{01}$ is unstable. \\
 $E_{02}$ is LAS. \\
 $E_{03}$ and $E_{04}$ do not exist.
\end{tabular}
& \begin{tabular}[c]{@{}l@{}} 
$E_{04}$ is LAS. \\
 $E_{01}, E_{02}, E_{03}$ are unstable.
\end{tabular} \\
\hline
\end{tabular}
\vspace{0.3em} % optional spacing
\caption{ Steady states of model~\eqref{eq:simplified_model_1} and their local stability under varying basic offspring number $N$ and the behavioral payoff difference $r_c - r_d$.}
\label{tab:LAS_summary_rigorous}
\end{table}

\color{black}
\noindent
\begin{remark}
The steady state and stability analysis of model~\eqref{eq:simplified_model_1}  extends to the case with constant public health interventions, as considered by Asfaw et al. \cite{asfaw_itn_control}. Formally, we may let $w$ evolve according to
\begin{equation*}
\frac{dw}{dt} = kw(1 - w)(-r_c + r_d) + \gamma(1 - w),
\end{equation*}
where $\gamma>0$ measures the effectiveness of such interventions. This modified system also admits four steady states, with local stability determined by critical thresholds for $r_c - r_d$ and $N$. A key difference is that $w$ can take the equilibrium value
$w^* = \frac{\gamma}{k (r_c - r_d)}$
if $ r_c-r_d \geq  \frac{\gamma}{k} $ (which ensures $0\leq w^* \leq 1$), indicating that sustained public health action can maintain a positive fraction of households performing breeding site control. A detailed analysis of this model is presented in the supplementary material.
\color{black}
\end{remark}

\subsection{The full game-theoretic model}
\label{full_GTM}
We analyze the proposed Equations~\eqref{eq:behavioral_entomological_system}, which include the perceived payoff for households not performing breeding site control given by the linear function $f_d(A_v) = - r_d m A_v$.
In Proposition \ref{prop: Game-theoretic Model}, we characterize the steady states of the system. The proof can be found in the supplementary material.

\begin{proposition}\label{prop: Game-theoretic Model}
The steady states $\left(L_v^*, A_v^*, w^*\right)$ for System~\eqref{eq:behavioral_entomological_system} consist of $E_{01}$, $E_{02}$, $E_{03}$, and $E_{04}$ characterized in Proposition \ref{prop: Baseline Model}, along with  $E_{05}$ representing a mosquito-positive state with partial breeding site control, given by
\begin{equation}
\left(L_v^*, A_v^*, w^*\right)=\left(\frac{\mu_A r_c}{r_dm\nu_L}, \frac{r_c}{r_dm}, \frac{K_{v_{\max }}-\frac{\mu_A r_c}{r_dm\nu_L\left(1-\frac{1}{N}\right)}}{K_{v_{\max }}-K_{v_{\min }}}\right)
\label{E05_formula}
\end{equation}

which exists if and only if $N>1$ and \textcolor{black}{$\frac{r_c}{r_d} \in\left[\alpha K_{v_{\min }}\left(1-\frac{1}{N}\right), \alpha K_{v_{\max }}\left(1-\frac{1}{N}\right)\right]$} where $\alpha=\frac{m \nu_L}{\mu_A}.$

\end{proposition}

 Here, the interval for {\( \frac{r_c}{r_d}\)} ensures that \( w^* \in [0,1] \). Remarkably, \( E_{05} \) depends not only on the entomological parameters but also on the perceived risks \( r_c \), \( r_d \), and the sensitivity parameter \( m \), reflecting a behavioral component that is absent from \( E_{01} \) through \( E_{04} \). Theorem~\ref{thm: Game-theoretic Model} summarizes the stability conditions for equilibria \( E_{01} \) through \( E_{04} \) (see the supplementary material for a proof).

\begin{theorem}\label{thm: Game-theoretic Model}
The local stability of $E_{01}$ -- $E_{04}$ for System~\eqref{eq:behavioral_entomological_system} is characterized as follows:
\begin{enumerate}
 \item \( E_{01} \) is LAS if \( N < 1 \) and unstable if \( N > 1 \).
    \item \( E_{02} \) is unstable.
    
    \item $E_{03}$ is LAS if  $\frac{r_c}{r_d} > \alpha K_{v_{\max}}\left(1 - \frac{1}{N}\right)$ and unstable if $\frac{r_c}{r_d} < \alpha K_{v_{\max}}\left(1 - \frac{1}{N}\right).$

    \item $E_{04}$ is LAS if  $\frac{r_c}{r_d} < \alpha K_{v_{\min}}\left(1 - \frac{1}{N}\right)$ and unstable if $\frac{r_c}{r_d} > \alpha K_{v_{\min}}\left(1 - \frac{1}{N}\right).$

\end{enumerate}
\end{theorem}

The stability conditions for \( E_{05} \) require a separate analysis. In short, we obtain a cubic characteristic equation 
\begin{equation}
\label{eq:cubic}
\lambda^3 + a_2 \lambda^2 + a_1 \lambda + a_0 = 0
\end{equation}
which does not admit a simple factorization. The Routh-Hurwitz criterion \cite{ogata2009modern}  guarantees that $E_{05}$ is LAS if and only if the following conditions are satisfied: 
\begin{equation*}
\ a_2 > 0, \ a_1 > 0, \ a_0 > 0 \quad \text{and} \quad  a_2 a_1 > a_0.
\end{equation*}

Here the coefficients are given by

$$\quad a_2=\frac{rb\nu_L \left(1-\frac{1}{N}\right)}{\mu_A}+\left(\nu_L+\mu_L+\mu_A\right), \qquad a_1=rb\nu_L\left(1-\frac{1}{N}\right) $$

and 
$$ a_0=P\left[\left(\alpha\left(1-\frac{1}{N}\right)K_{v_{\max }}-\frac{r_c}{r_d}\right)\left(-\alpha\left(1-\frac{1}{N}\right)K_{v_{\min }}+\frac{r_c}{r_d}\right)\right]$$ with $P=\frac{krbr_d\mu_A}{m\left(K_{v_{\max }}-K_{v_{\min }}\right)}$.  The first two conditions $a_2>0$ and $a_1>0$ are clearly satisfied, given that $N>1$. Moreover, the condition $a_0>0$ is equivalent to $\alpha K_{v_{\min }}\left(1-\frac{1}{N}\right) \leq \frac{r_c}{r_d} \leq \alpha K_{v_{\max }}\left(1-\frac{1}{N}\right)$, which is required for the existence of the steady state $E_{05}$. 
\color{black}
For the condition $a_2 a_1 > a_0$, we can see that 
\[
a_2 a_1 - a_0 = \left(\frac{r_c}{r_d}\right)^2 
- \alpha \left(1 - \frac{1}{N}\right) \left(K_{v_{\max}} + K_{v_{\min}}\right) \left(\frac{r_c}{r_d}\right) 
+ \alpha^2 \left(1 - \frac{1}{N}\right)^2 K_{v_{\max}} K_{v_{\min}} 
+ Q \alpha \left(K_{v_{\max}} - K_{v_{\min}}\right)\left(1 - \frac{1}{N}\right)
\]

and thus $a_2 a_1 - a_0$ takes the quadratic form   $f(x) = x^2 - Bx + C$, with $ x = \frac{r_c}{r_d}$ and  coefficients given by
\begin{equation} 
\label{Bexpression}
    B = \alpha \left(1 - \frac{1}{N}\right) \left(K_{v_{\max}} + K_{v_{\min}}\right), 
\end{equation}
and 
\begin{equation} 
\label{Cexpression}
    C = \alpha^2 \left(1 - \frac{1}{N}\right)^2 K_{v_{\max}} K_{v_{\min}} 
    + Q \alpha \left(K_{v_{\max}} - K_{v_{\min}}\right) \left(1 - \frac{1}{N}\right), \quad \text{where} \ Q = \tfrac{\mu_A^2 + r b \nu_L}{k \mu_A r_d}.
\end{equation}  

Hence, the inequality $a_2 a_1 > a_0$ is equivalent to $f(x) > 0$. Theorem~\ref{thm: Game-theoretic Model_} summarizes this result (see supplementary material for the detailed proof).

\color{black}

\begin{theorem}\label{thm: Game-theoretic Model_}

Let $E_{05}$ be given by Equation~\eqref{E05_formula} and   $f$ be the quadratic polynomial $f(x) = x^2 - Bx + C$, where the coefficients \( B \) and \( C \) are given by Equations ~\eqref{Bexpression} and ~\eqref{Cexpression}. Then, the following conditions hold:

\begin{enumerate}

\item[{(i)}] If $B^2-4 C<0$, then $E_{05}$ is LAS within the interval: 
$\alpha\left(1-\frac{1}{N}\right) K_{v_{\min }}<\frac{r_c}{r_d}<\alpha\left(1-\frac{1}{N}\right) K_{v_{\max }}$.

\item[{(ii)}] If \( B^2 - 4C \geq 0 \)
then $E_{05}$ is LAS if
\[
\alpha \left(1 - \frac{1}{N} \right) K_{v_{\min}} < \frac{r_c}{r_d} < x_1 \quad 
\text{or} \quad
x_2 < \frac{r_c}{r_d} < \alpha \left(1 - \frac{1}{N} \right) K_{v_{\max}}.
\]

where $x_1$ and $x_2$ are the real roots of $f(x).$

\end{enumerate}
\end{theorem}

Condition $(ii)$ in Theorem \ref{thm: Game-theoretic Model_} yields the parameter regime under which the steady state $E_{05}$ may lose stability, namely $\frac{r_c}{r_d} \in (x_1,x_2)$. In particular, for a Hopf bifurcation to occur, the cubic characteristic Equation~\eqref{eq:cubic} must admit purely imaginary roots. This condition can be achieved at the Routh–Hurwitz boundary $a_0=a_2 a_1$. In fact,  substituting $a_0=a_2 a_1$ into Equation~\eqref{eq:cubic} yields
$$
 \lambda^3+a_2 \lambda^2+a_1 \lambda+a_2 a_1=0 
\Rightarrow \left(\lambda+a_2\right)\left(\lambda^2+a_1\right)=0.
$$
The roots in this case are $\lambda_1=-a_2$ and $\lambda_{2,3}= \pm i \sqrt{a_1}$ which implies $\lambda_{2,3}= \pm i \sqrt{r b \nu_L\left(1-\frac{1}{N}\right)}.$  For the quadratic polynomial $f$, the condition $a_0 = a_2 a_1$ is equivalent to $f(x) = 0$, which yields

\[
x_{1,2} = \frac{r_c}{r_d} =  \frac{B \pm \sqrt{B^2 - 4C}}{2}.
\]

The next theorem confirms that a Hopf bifurcation occurs at $E_{05}$ under the above condition. 

\begin{theorem}[Hopf Bifurcation]
\label{hopf}
For \(B\) and \(C\) defined by Equations~\eqref{Bexpression} and \eqref{Cexpression}, such that \(B^2 - 4C> 0\), a Hopf Bifurcation occurs at \(E_{05}=\left(L_v^*, A_v^*, w^*\right)\) given by Equation~\eqref{E05_formula} if
\[
\frac{r_c}{r_d}
= \alpha \frac{K_{v_{\max}}+K_{v_{\min}}}{2}\left(1-\frac{1}{N}\right) 
\pm \frac{\sqrt{B^2-4C}}{2}.
\]

\end{theorem}

\begin{proof} 

Here, we choose the imitation rate \(k\,\) as our bifurcation parameter. Let \(\lambda_1(k)\) and \(\lambda_2(k)\) be the two complex conjugate eigenvalues of the characteristic Equation~\eqref{eq:cubic}. If \(\frac{r_c}{r_d}
= \alpha \frac{K_{v_{\max}}+K_{v_{\min}}}{2}\left(1-\frac{1}{N}\right) 
\pm \frac{\sqrt{B^2-4C}}{2},\) then we have \(a_0=a_2a_1\) and hence, \(\mathbb{R}e(\lambda_1(k))=\mathbb{R}e(\lambda_2(k))=0\).  To prove the transversality condition, we implicitly differentiate $p(\lambda(k), k)=\lambda^3+a_2(k) \lambda^2+a_1(k) \lambda+a_0(k) \equiv 0$ with respect to $k$ and obtain $$p_\lambda(\lambda, k) \lambda^{\prime}(k)+p_k(\lambda, k)=0 \Rightarrow \lambda^{\prime}(k)=-\frac{p_k(\lambda, k)}{p_\lambda(\lambda, k)}
.$$
As noted before, the coefficients of the characteristic equation are given by

$$
a_2=\frac{r b \nu_L(1-1 / N)}{\mu_A}+\left(\nu_L+\mu_L+\mu_A\right), \quad a_1=r b \nu_L(1-1 / N),$$ 
$$\quad a_0=P\left[\left(\alpha(1-1 / N) K_{v_{\max }}-\frac{r_c}{r_d}\right)\left(-\alpha(1-1 / N) K_{v_{\min }}+\frac{r_c}{r_d}\right)\right],
$$

where $P=\frac{k r b r_d \mu_A}{m\left(K_{v_{\text {max }}}-K_{v_{\text {min }}}\right)}$.
Since only $a_0$ depends linearly on $k$ through the term $P$, we have $a_2^{\prime}(k)=a_1^{\prime}(k)=0$ and $a_0^{\prime}(k) > 0$.

Hence, $p_\lambda=3 \lambda^2+2 a_2 \lambda+a_1$, and $p_k=a_2^{\prime} \lambda^2+a_1^{\prime} \lambda+a_0^{\prime}=a_0^{\prime}(k)$. We substitute these values and write
\begin{equation} 
\label{eq:lambda-derivative}
\lambda^{\prime}(k) \;=\; -\,\frac{ a_0^{\prime}(k)}
{3\lambda^2 + 2a_2 \lambda + a_1}.
\end{equation}

Now suppose when \(k=k_c\), we get  \(a_0=a_2a_1\), or equivalently, \(\frac{r_c}{r_d}
= \alpha \frac{K_{v_{\max}}+K_{v_{\min}}}{2}\left(1-\frac{1}{N}\right) 
\pm \frac{\sqrt{B^2-4C}}{2}\).  Hence, we obtain $\lambda_{2,3}\left(k_c\right)= \pm i \omega= \pm i \sqrt{a_1}.$
By substituting these values in \eqref{eq:lambda-derivative}, we have

$$
\lambda^{\prime}\left(k_c\right)=\frac{a_0^{\prime}\left(k_c\right)}{2 a_1 \pm 2 a_2 i \sqrt{a_1}}.
$$

Taking the real part and simplifying yields

$$
\left.\frac{d}{d k} \operatorname{Real}(\lambda(k))\right|_{k=k_c}=\frac{a_0^{\prime}\left(k_c\right)}{2\left(a_1+a_2^2\right)} \neq 0,
$$

and since $a_1>0$, the denominator is positive and at $k=k_c, \,a_0^{\prime}\left(k_c\right) \neq 0$. Hence, the transversality condition holds.

\end{proof}

\color{black}
\begin{remark}
The results from Theorems \ref{thm: Game-theoretic Model_} and \ref{hopf} are connected by the representation of the quadratic function $f$ 
in terms of $\frac{r_c}{r_d}$ and~$k$ as independent variables:
\[
f\!\left(\frac{r_c}{r_d}, k\right)
= \left(\frac{r_c}{r_d}\right)^2
- B \left(\frac{r_c}{r_d}\right)
+ C(k).
\]
Here $B$ is given by Equation~\ref{Bexpression} and $C(k)$ corresponds to Equation~\eqref{Cexpression} expressed with the explicit dependence on the imitation rate $k$, i.e.,
\[
C(k)
= \alpha^2\!\left(1-\frac{1}{N}\right)^{\!2} K_{v_{\max}} K_{v_{\min}}
+ Q(k)\,\alpha\!\left(K_{v_{\max}} - K_{v_{\min}}\right)\!\left(1-\frac{1}{N}\right),
\ \text{where} \  
Q(k)
= \left(\frac{\mu_A^2 + r b \nu_L}{ \mu_A r_d} \right) \frac{1}{k}.
\]
\
Hence, $f$ is quadratic with respect to $\frac{r_c}{r_d}$ and decreasing with respect to $k$. At a bifurcation point $k = k_c$, we know that $f(r_c/r_d, k_c) = 0$. Therefore, when $k < k_c$, we have
\[
f\!\left(\frac{r_c}{r_d}, k\right)
> f\!\left(\frac{r_c}{r_d}, k_c\right) = 0,
\]
and $E_{05}$ is LAS. Conversely, for $k > k_c$, then $f\!\left(\frac{r_c}{r_d}, k\right)
< 0$ and $E_{05}$ becomes unstable.
\end{remark}

\color{black}

\section{Numerical simulations} 
\label{numerics}

We complement the stability analysis with numerical simulations of the proposed models. Numerical integration of our systems was done using MATLAB's ode45 solver (4th/5th-order Runge-Kutta-Fehlberg method). The simulations were intended to illustrate and support the theoretical results rather than provide exhaustive parameter exploration.

\subsection*{Simplified model with constant payoffs}

Numerical simulations for  System~\eqref{eq:simplified_model_1} are depicted in Figure~\ref{fig:model1}. The central colormap indicates the steady state that is achieved when the system is simulated for a range of basic offspring numbers ($N$) and perceived risk differences ($r_c-r_d$). Four representative time series from each parameter region are displayed around the colormap. Dashed lines indicate the analytically derived steady states given in Proposition~\ref{prop: Baseline Model}. The results align with our theoretical estimates. For example, the top-left rectangle (light blue) in the colormap displays the parameter region for which the steady state $E_{01}$ (``Mosquito-free, no control") is locally asymptotically stable (LAS) according to Theorem~\ref{thm: Baseline Model}. Similarly,  the bottom-left purple region corresponds to the conditions where $E_{02}$  (``Mosquito-free, full control") is LAS. In this regime, the mosquito population declines to extinction since $N<1$, while the household control proportion $w(t)$ converges to $1$ because $r_c-r_d<0$. The top and bottom-right rectangles (dark yellow and red) correspond to the case where the steady states $E_{03}$ and $E_{04}$ are LAS. In both cases, the mosquito populations converge to a positive steady state due to $N>1$ while the household control proportion $w(t)$ converges to zero or one depending on $r_c-r_d$.

\begin{figure}[H]  % change [h] to [t] for top, [b] for bottom, etc.
  \centering
  \includegraphics[width=\textwidth]{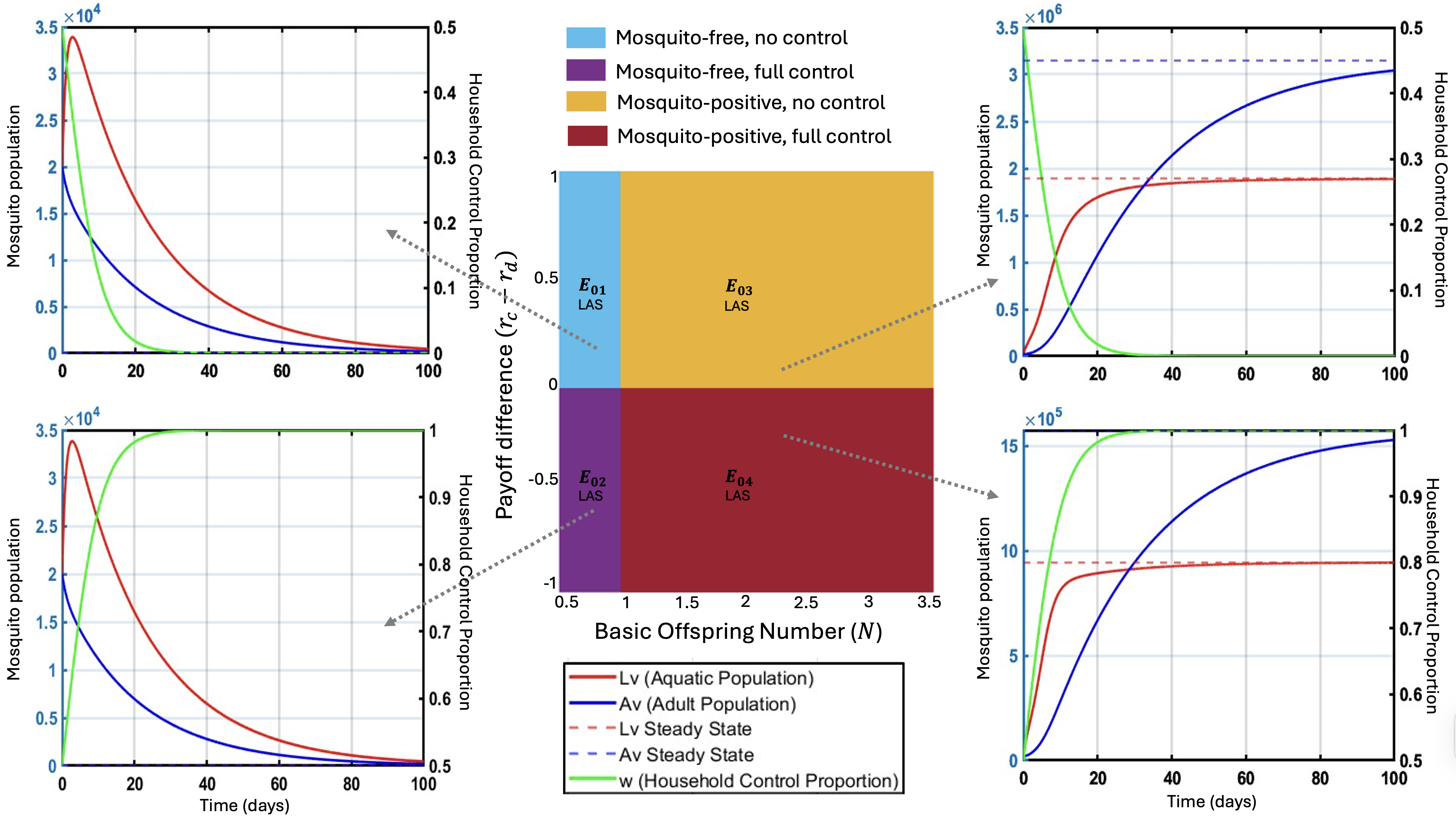}  % Adjust width as needed
  \caption{Parameter regions of stability and sample trajectories for the simplified System~\eqref{eq:simplified_model_1} with constant payoffs. The central colormap shows the four regions in which the steady states $E_{01} - E_{04}$ are LAS. For each region, we exhibit trajectories converging to the corresponding steady state. Parameter values taken from \cite{dumont_mosquito_control}: $r = 0.5$, $\nu_L   = 0.067\,\text{days}^{-1}$, $\mu_L = 0.62\,\text{days}^{-1}$, and $\mu_A =  0.04\,\text{days}^{-1}$. Other parameters set to plausible values: $K_{v_{\max}} = 2\times 10^{6}$, $K_{v_{\min}} = 1\times 10^{6}$, $t_{\text{span}} = [0,100]$ days, $k = 0.8\,\text{days}^{-1}$, and $b$ ranging from $1$ to $15$ to generate different $N$ values. Initial conditions: $L_0 = 20000$, $A_0 = 20000$, and $w_0 = 0.5$.
}
  \label{fig:model1} 
\end{figure}
\subsection*{Full game-theoretic model}

In Figure~\ref{fig3}, we present numerical simulations for the full game-theoretic model~\eqref{eq:behavioral_entomological_system}. We examine the plane $N$ vs $\frac{rc}{rd}$, which is partitioned into regions where each steady state is LAS, along with the area where $E_{05}$ is unstable. Representative trajectories illustrating convergence to a steady state or sustained oscillations are shown.  The boundaries separating the parameter regions are derived from our analytical results. The left portion of the colormap corresponds to the region $N<1$ for which the mosquito-free steady state with no breeding site control $\left(E_{01}\right)$ is LAS. The top (dark yellow) and bottom (dark red) colormap regions when $N>1$ correspond to the parameter regimes for which $E_{03}$ or $E_{04}$ are LAS, respectively. These two regions are separated by the region where the interior equilibrium $E_{05}$ exists and may be stable or unstable.  When $E_{05}$ is LAS (light yellow region), we observe damped oscillations converging to a state where mosquito populations and household control remain at intermediate levels (``mosquito-positive, partial control"). When $E_{05}$ is unstable (central off-white region), the model trajectories exhibit sustained oscillations arising through the Hopf bifurcation (Theorem~\ref{hopf}). 
 
 The light yellow region is bounded by the two curves given explicitly by
\[
  \frac{r_c}{r_d} = \alpha K_{v_{\min}}\!\left(1-\frac{1}{N}\right)
  \quad \text{and} \quad
  \frac{r_c}{r_d} = \alpha K_{v_{\max}}\!\left(1-\frac{1}{N}\right).
\]
The interior equilibrium \(E_{05}\) exists only when \(\frac{r_c}{r_d}\) lies strictly between these two boundary values (see Proposition~\ref{prop: Game-theoretic Model}).The off-white region corresponds to the parameter values for which condition (ii) in Theorem~\ref{thm: Game-theoretic Model_} is satisfied, ensuring that \(E_{05}\) is unstable. This region would not appear if \(B^2 - 4C < 0\), since in that case condition (i) of Theorem~\ref{thm: Game-theoretic Model_} would prevent any instability of \(E_{05}\).

\begin{figure}[H]  
  \centering
\includegraphics[width=\textwidth]{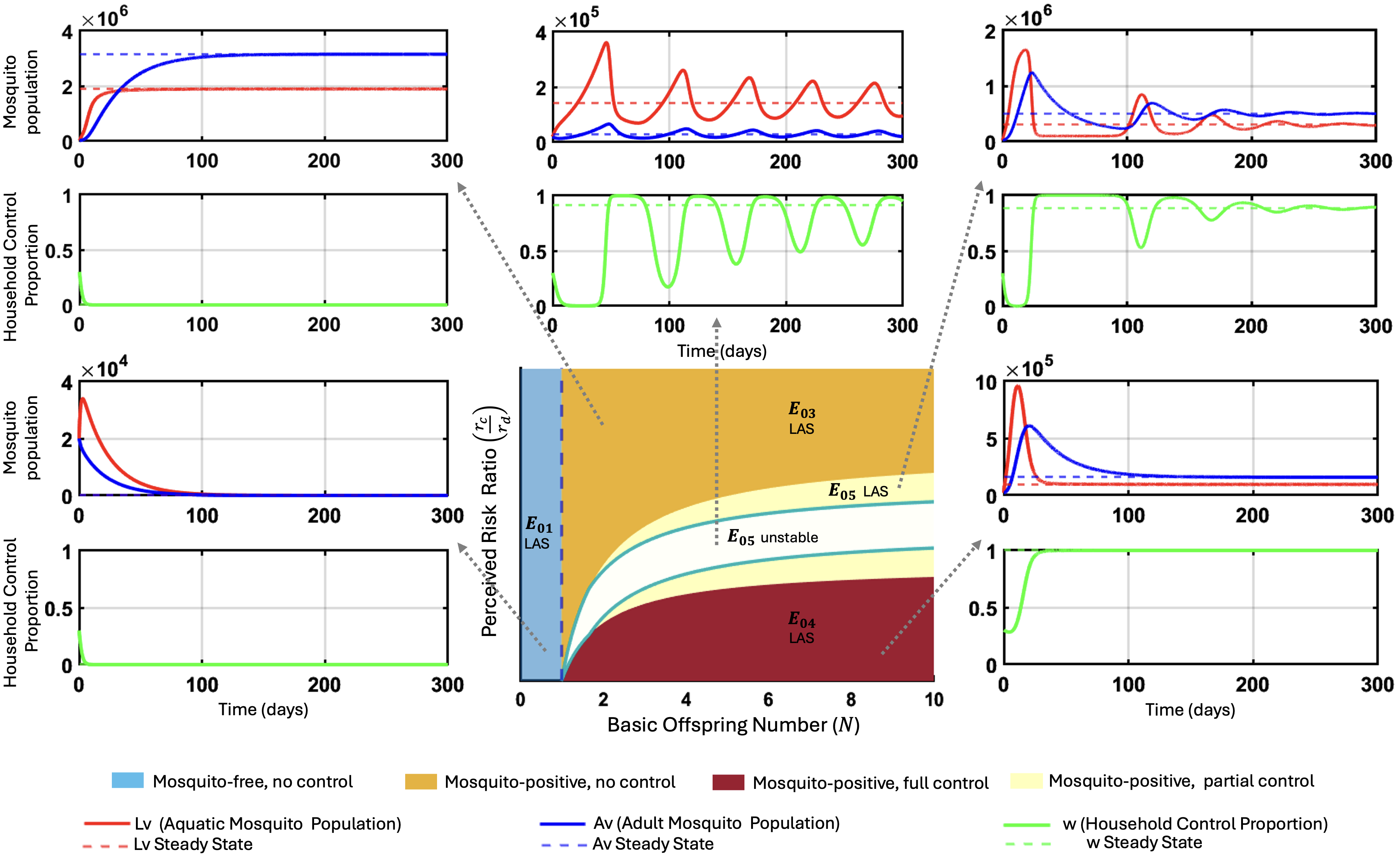} 
  \caption{Parameter regions of stability and sample trajectories of the full game-theoretic model~\eqref{eq:behavioral_entomological_system}. The colormap highlights four distinct regions where steady states are LAS and one where $E_{05}$ is unstable and oscillations emerge. For each region, corresponding system trajectories converge to a steady state or oscillate. Parameter values taken from \cite{dumont_mosquito_control}: $r = 0.5$,  $\nu_L   = 0.067\,\text{days}^{-1}$, $\mu_L = 0.62\,\text{days}^{-1}$, and $\mu_A =  0.04\,\text{days}^{-1}$. Other parameters set to plausible values: $K_{v_{\max}} = 2\times 10^{6}$, $K_{v_{\min}} = 1\times 10^{5}$, $t_{\text{span}} = [0,300]$ days, $k \in[0.5,0.8]$ days $^{-1}$, $m = 0.3\,\text{mosquito}^{-1}$, and $b$ ranging from $1$ to $15$ to generate different $N$ values. Initial conditions: $L_0 = 20000$, $A_0 = 20000$, and $w_0 = 0.3$.}  
  \label{fig3}   
\end{figure}

\textbf{Amplitude and period of oscillations}

We investigate the amplitude and period of the sustained oscillations that arise in the unstable equilibrium region for \(E_{05}\). Figure~\ref{fig4} shows a subregion of the off-white instability zone in Figure~\ref{fig3}, corresponding to basic offspring number values approximately between \(1.4\) and \(2.8\). This region is depicted in two panels colored according to the amplitude and period of oscillations in the aquatic mosquito population, computed for each pair \((N,\, r_c/r_d)\). Four representative trajectories---with periods of approximately 50, 86, 109, and 165 days---are displayed around the colormap. Larger values of \(N\) produce higher amplitudes and shorter periods, reflecting the enhanced reproductive capacity of mosquitoes and the faster response from household control. Conversely, for lower values of \(N\), oscillations have smaller amplitudes and longer periods, corresponding to a slower growth--and--control cycle between the mosquito population and household behavior. As the pair \((N,\, r_c/r_d)\) approaches either Hopf boundary, both the amplitude and the period of oscillations decrease to zero. This behavior indicates proximity to the stability threshold at which the limit cycle collapses, and trajectories return to a stable steady state. Both the amplitudes and periods were computed using MATLAB’s built-in function \emph{findpeaks}.

\begin{figure}[H]  
  \centering
  \includegraphics[width=\textwidth]{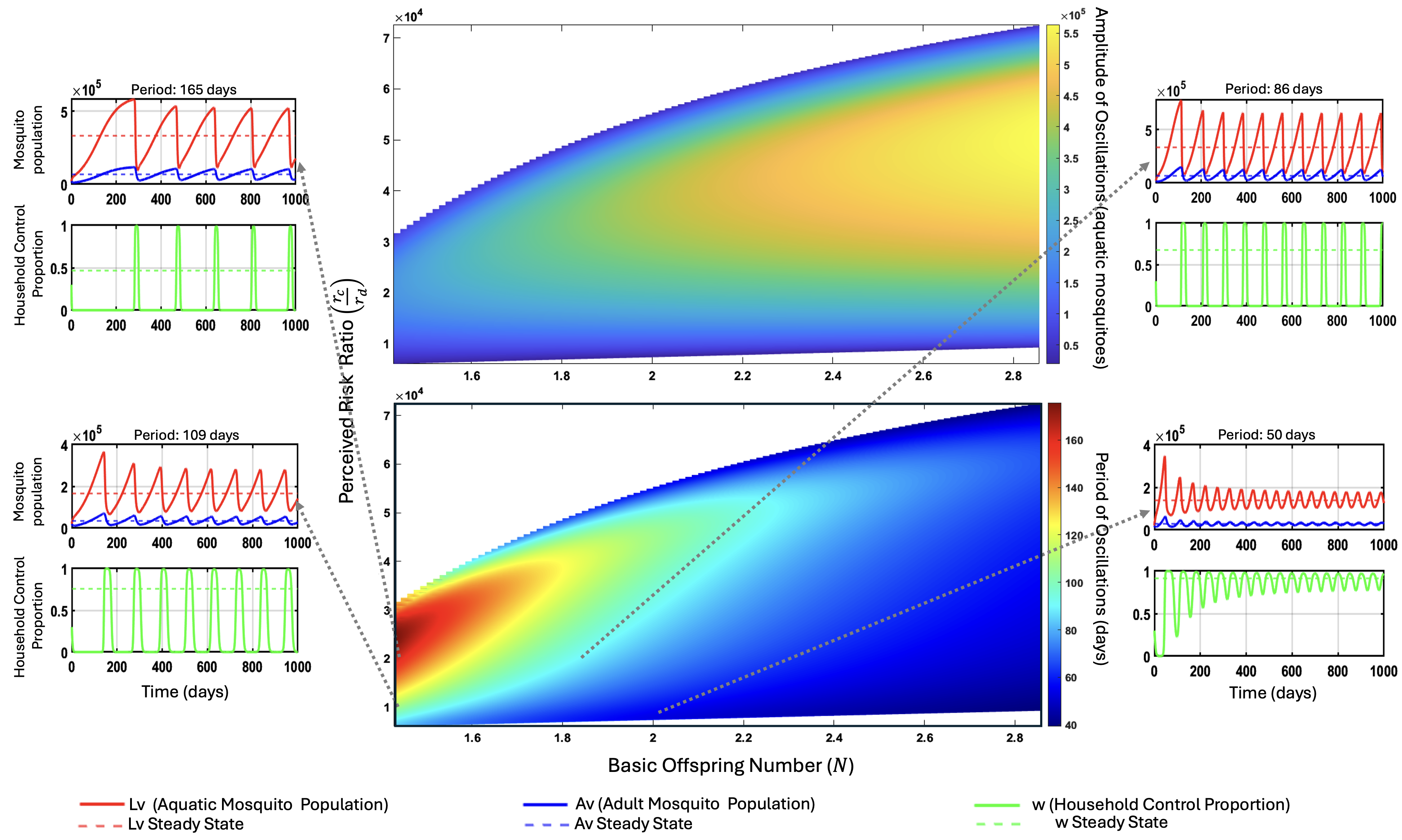}  
  \caption{Amplitude and period of sustained oscillations in mosquito populations and household behavior. Oscillations in aquatic mosquitoes ($L_v$), adult mosquitoes ($A_v$), and household mechanical control ($w$) vary with the basic offspring number $N$ and the perceived risk ratio $r_c/r_d$. Larger values of $N$ produce oscillations with higher amplitudes and shorter periods, while smaller values of $N$ result in lower amplitudes and longer periods. Parameter values: $r = 0.5$, $b$ ranging from $1$ to $15$ to generate different $N$ values,  $\nu_L   = 0.04\,\text{days}^{-1}$, $\mu_L = 0.03\,\text{days}^{-1}$, and $\mu_A =  0.2\,\text{days}^{-1}$. Other parameters set to plausible values: $K_{v_{\max}} = 2\times 10^{6}$, $K_{v_{\min}} = 1\times 10^{5}$, $t_{\text{span}} = [0,1000]$ days, $k = 0.8\,\text{days}^{-1}$, $m = 0.3\,\text{mosquito}^{-1}$. Initial conditions: $L_0 = 20000$, $A_0 = 20000$, and $w_0 = 0.3$.}  
  \label{fig4}  
\end{figure}

\section{Conclusions} 
\label{conclusions}

We developed a game-theoretic model of mosquito population dynamics under household control of breeding sites. Our framework integrates mosquito development with strategic human responses through (i) a behavior-dependent carrying capacity that linearly decreases as more households perform breeding site control ($K_{v} = K_v(w)$), and (ii) a perceived payoff for not performing control that may depend on the mosquito prevalence ($f_d = f_d(A_v)$). When such a payoff is considered constant, the behavior component of our model is decoupled from the mosquito dynamics, and the fraction of households performing control converges to either 0 or 1. Further, if the basic offspring number $N$ is greater than 1, both the aquatic and adult populations converge to positive steady states. In contrast, the full game-theoretic model with prevalence-dependent payoffs admits richer dynamics as the system undergoes a Hopf bifurcation, generating sustained oscillations in mosquito abundance and household behavior.

A feedback loop can explain the oscillatory regimes observed in the full game-theoretic model. When adult mosquito abundance increases, households intensify breeding site control, thereby reducing the carrying capacity and mosquito populations. As mosquito abundance declines, the perceived risk of not performing control also decreases, leading households to relax their efforts. This relaxation allows the mosquito population to recover, initiating a new cycle of resurgence and renewed control. Within the parameter region where oscillations are possible, their amplitude and period vary with the basic offspring number $N$ and the perceived risk ratio $r_c/r_d$ (Figure~\ref{fig4}). Larger values of $N$ produce faster mosquito population growth, eliciting a sharper but shorter-lived behavioral response. This process generates oscillations with higher amplitude and shorter periods. Conversely, as $N$ approaches $1$ from the right, mosquito growth is slower, households adjust more gradually, and the resulting oscillations have smaller amplitudes and longer periods.

Sustained oscillations in mosquito population models have been documented across a variety of frameworks. Several studies have shown that incorporating time delays or environmental drivers---such as temperature or precipitation effects on oviposition, development, or mortality---can lead to Hopf bifurcations and persistent oscillatory dynamics \cite{abdelrazec2017mathematical,okuneye2018mathematical,ngwa2010mathematical,cai2018dynamics}. Oscillations have also been reported in models that incorporate nonlinearities, such as saturation effects in sterile-mosquito release strategies \cite{cai2014dynamics,huang2019bifurcation,hussein2025residual,wang2024periodic}.  Strugarek et al.\ demonstrated that even in the absence of seasonal forcing or temperature-driven rates, a model with a larval density-mediated hatching function can generate stable oscillations \cite{strugarek2019oscillatory}. In the context of time-varying carrying capacities, as we considered in this work, Dumont and Thuilliez \cite{dumont_mosquito_control} showed that interventions may induce spiked periodic behavior through impulsive differential equations. Here we demonstrated that sustained oscillations can also emerge through a feedback mechanism typically observed in the literature of game-theoretic models \cite{bauch2005imitation, chang2020game, huang2022game}, arising from the coupling of linear functional forms of a behavior-dependent carrying capacity and a prevalence-dependent payoff.

Our study has limitations. Realistic descriptions of mosquito population dynamics require incorporating environmental factors such as temperature-dependent development and survival rates \cite{lana2018assessment} or rainfall-driven creation of larval habitats \cite{abdelrazec2017mathematical}. Extending our framework to include such weather-sensitive parameters represents a promising research direction, as the combined influence of environmental drivers and behavioral feedback may reproduce oscillatory behavior observed in empirical data. Our analysis of the Hopf bifurcation in the full game-theoretic model is restricted to local stability conditions derived from the Routh--Hurwitz criteria. A full characterization of the bifurcation requires computing the first Lyapunov coefficient using normal form theory \cite{kuznetsov1998elements, wiggins2003introduction}. While our simulations indicate that trajectories converge smoothly to stable periodic orbits, establishing their existence and stability rigorously remains a direction for future work. The model assumes a homogeneous population in which all households respond identically to perceived mosquito risk. In reality, behavioral responses vary across socioeconomic groups, neighborhoods, and levels of information access. Incorporating heterogeneity or multiple behavioral classes could alter the onset or amplitude of oscillations. A multi-patch or network formulation could reveal spatially asynchronous oscillations or localized behavioral effects.

\section{Acknowledgments} 
\label{acknowledgments}

The authors thank Haridas Das and Abdullah Al Helal for their careful feedback on this work, and Nick Belden for his proofreading of the manuscript. L.M.S. acknowledges support from the National Science Foundation under award DMS-2327844.

\bibliographystyle{unsrt}  
\bibliography{scibib}  
\newpage

\section{Supplementary material}

\subsection{Proofs of propositions and theorems}

\begin{proof}[\textbf{Proof of Proposition \ref{prop: Baseline Model}}]

By setting the behavioral equation equal to zero, we have \(w = 0\), \(w = 1\), and $r_c=r_d$. However, $r_c=r_d$ does not provide a typical steady state for \(w\). Therefore, the solutions are \(w = 0\) and \(w = 1\).

Substitute \(A_v = \frac{\nu_L}{\mu_A} L_v\) and \(K_v(w) = K_{v_{\max}} - w (K_{v_{\max}} - K_{v_{\min}})\), we obtained \[
r b \left(\frac{\nu_L}{\mu_A} L_v \right) \left(1 - \frac{L_v}{K_v(w)}\right) - (\nu_L + \mu_L) L_v = 0
\]

i.e., \[
L_v \left[\left(\frac{r b \nu_L}{\mu_A}\right) \left(1 - \frac{L_v}{K_v(w)}\right) - (\nu_L + \mu_L)\right] = 0.
\]

This gives two solutions \[L_v = 0\] or\[
\left(\frac{r b \nu_L}{\mu_A}\right) \left(1 - \frac{L_v}{K_v(w)}\right) = \nu_L + \mu_L
\]

i.e.,  $$
\begin{aligned}
L_v & =K_v(w)\left(1-\frac{\mu_A\left(\nu_L+\mu_L\right)}{r b \nu_L}\right) \\
\Rightarrow L_v & =K_v(w)\left(1-\frac{1}{N}\right) \text { if and only if } N>1 .
\end{aligned}
$$

Hence, $A_v=\frac{\nu_L}{\mu_A} L_v=\frac{\nu_L}{\mu_A} K_v(w)\left(1-\frac{1}{N}\right).$

When $w=0, L_v=0$ then we get $A_v=0.$ Moreover, the following holds:

\begin{itemize}
\item When $w=1, L_v=0$ then we get $A_v=0.$

\item When $w=0$, we get $L_v=K_v(0)\left(1-\frac{1}{N}\right)=K_{v_{\text {max }}}\left(1-\frac{1}{N}\right)$ and
$
A_v=\frac{\nu_L}{\mu_A} K_v(0)\left(1-\frac{1}{N}\right)=\frac{\nu_L}{\mu_A} K_{v_{\max }}\left(1-\frac{1}{N}\right).
$

\item When $w=1$, we get $L_v=K_v(1)\left(1-\frac{1}{N}\right)=K_{v_{\text {min }}}\left(1-\frac{1}{N}\right)$ and
$
A_v=\frac{\nu_L}{\mu_A} K_v(1)\left(1-\frac{1}{N}\right)=\frac{\nu_L}{\mu_A} K_{v_{\min }}\left(1-\frac{1}{N}\right).
$
\end{itemize}

Therefore, we obtained the four steady states given by

\begin{enumerate}
\item \underline{\text{$E_{01}$:}} $\left(L_v^*, A_v^*, w^*\right)=(0,0,0),$

 \item \underline{\text{$E_{02}$:}} $\left(L_v^*, A_v^*, w^*\right)=(0,0,1),$

 \item \underline{\text{$E_{03}$:}} $\left(L_v^*, A_v^*, w^*\right)=\left(K_{v_{\text {max }}}\left(1-\frac{1}{N}\right), \frac{\nu_L}{\mu_A} K_{v_{\text {max }}}\left(1-\frac{1}{N}\right), 0\right)$ if and only if $N>1,$

\item  \underline{\text{$E_{04}$:}} $\left(L_v^*, A_v^*, w^*\right)=\left(K_{v_{\text {min }}}\left(1-\frac{1}{N}\right), \frac{\nu_L}{\mu_A} K_{v_{\text {min }}}\left(1-\frac{1}{N}\right), 1\right)$ if and only if $N>1.$
\end{enumerate}

\end{proof}

\begin{proof} [\textbf{Proof of Proposition \ref{prop: Game-theoretic Model}}]

We have $\nu_L L_v=\mu_A A_v \Longrightarrow A_v=\frac{\nu_L}{\mu_A} L_v$ and, by setting the derivative to zero, we have $k  w(1-w)\left[-r_c+r_d  m  A_v\right]=0$, which implies,  $w=0, w=1$ and $A_v=\frac{r_c}{r_d  m}$.

By substituting $\mathrm{A}_v=\frac{\nu_L}{\mu_A} L_v$ and $w=0$, we obtain

$$
 rb\left(\frac{\nu_L}{\mu_A} L_v\right)\left(1-\frac{L_v}{K_v(0)}\right)-\left(\nu_L+\mu_L\right) L_v=0 
\Rightarrow  L_v\left[\left(\frac{r  b  \nu_L}{\mu_A}\right)\left(1-\frac{L_v}{K_{v_{\max }}}\right)-\left(\nu_L+\mu_L\right)\right]=0
$$
So, we have either $L_v=0$ or
$$
\left(\frac{rb\nu_L}{\mu_A}\right)\left(1-\frac{L_v}{K_{v_{\max }}}\right)=\left(\nu_L+\mu_L\right)
\Rightarrow \quad  L_v=K_{v_{\max }}\left(1-\frac{1}{N}\right) \text { if and only if } N>1.
$$
Again, by substituting $A_v=\frac{\nu_L}{\mu_A} L_v$ and $w=1$, we obtain two more steady states. Hence, the first four steady states are given by

\begin{enumerate}
\item \underline{\text{$E_{01}$:}} $\left(L_v^*, A_v^*, w^*\right)=(0,0,0),$ 

\item \underline{\text{$E_{02}$:}} $\left(L_v^*, A_v^*, w^*\right)=(0,0,1),$

\item \underline{\text{$E_{03}$:}} $\left(L_v^*, A_v^*, w^*\right)=\left(K_{v_{\text {max }}}\left(1-\frac{1}{{N}}\right), \frac{\nu_L}{\mu_A} K_{v_{\text {max }}}\left(1-\frac{1}{{N}}\right), 0\right)$ if and only if ${N}>1$ and

\item \underline{\text{$E_{04}$:}} $\left(L_v^*, A_v^*, w^*\right)=\left(K_{v_{\text {min }}}\left(1-\frac{1}{N}\right), \frac{\nu_L}{\mu_A} K_{v_{\min }}\left(1-\frac{1}{{N}}\right), 1\right)$ if and only if ${N}>1$.
\end{enumerate}

For $E_{05}$, we substitute  $A_v=\frac{r_c}{r_d  m}$ and $L_v=\frac{\mu_A r_c}{r_d m \nu_L}$, and obtain

$$
\begin{aligned}
& \quad r  b  \mathrm{~A}_v\left(1-\frac{L_v}{K_v(w)}\right)-\left(\nu_L+\mu_L\right) L_v=0 \\
& \Rightarrow r b \frac{r_c}{r_d m}\left(1-\frac{\mu_A r_c}{r_d m \nu_L k_v(w)}\right)-\left(\nu_L+\mu_L\right) \frac{\mu_A r_c}{r_d m \nu_L}=0 \\
& \Rightarrow 1-\frac{\mu_A  r_c}{r_d m \nu_L  k_v(w)}=\frac{\left(\nu_L+\mu_L\right) \mu_A}{\nu_L r b} \\
& \Rightarrow \frac{\mu_A  r_c}{r_d  m \nu_L k_v(w)}=1-\frac{1}{N} \\
& \Rightarrow k_v(w)=\frac{\mu_A r_c}{r_d m\nu_L\left(1-\frac{1}{N}\right)} \\
& \Rightarrow K_{v_{\max }}-w\left(K_{v_{\max }}-K_{v_{\min }}\right)=\frac{\mu_A  r_c}{r_d  m \nu_L\left(1-\frac{1}{N}\right)} \\
& \Rightarrow w=\frac{K_{v_{\max }}-\frac{\mu_A r_c}{r_d  m  \nu_L\left(1-\frac{1}{N}\right)}}{\left(K_{v_{\max }}-K_{v_{\min }}\right)} \text { if and only if} \quad N>1 \text \quad {\text{and}}
\end{aligned}
$$

$$
\begin{gathered}
0 \leq \frac{K_{v_{\max }}-\frac{\mu_A r_c}{r_d m \nu_L\left(1-\frac{1}{N}\right)}}{\left(K_{v_{\max }}-K_{v_{\min }}\right)} \leq 1 \\
\Rightarrow \alpha K_{v_{\min }}\left(1-\frac{1}{N}\right) \leq \frac{r_c}{r_d} \leq \alpha K_{v_{\max }}\left(1-\frac{1}{N}\right) \text { where } \alpha=\frac{m \nu_L}{\mu_A}.
\end{gathered}
$$

Hence, $E_{05}$ is given by

$$\left(L_v^*, A_v^*, w^*\right)=\left(\frac{\mu_A  r_c}{r_d  m \nu_L}, \frac{r_c}{r_d  m}, \frac{K_{v_{\max }}-\frac{\mu_A  r_c}{r_d  m \nu_L\left(1-\frac{1}{N}\right)}}{\left(K_{v_{\max }}-K_{v_{\min }}\right)}\right)$$ if and only if $N>1$ and
$\alpha K_{v_{\min }}\left(1-\frac{1}{N}\right) \leq \frac{r_c}{r_d} \leq \alpha K_{v_{\max }}\left(1-\frac{1}{N}\right)$ where $\alpha=\frac{m \nu_L}{\mu_A}.$

\end{proof}

\begin{proof}[\textbf{Proof of Theorem \ref{thm: Game-theoretic Model}}]

 We proceed with the analysis of each steady state separately, proving the claims from items 1--4.

\begin{itemize}

\item \underline{$E_{01}$:} We evaluate the Jacobian matrix at the steady state $E_{01}$. The Jacobian matrix at $E_{01}$ is given by

\[
J = \begin{bmatrix}
- (\nu_L + \mu_L) & r b & 0 \\
\nu_L & -\mu_A & 0 \\
0 & 0 & -k r_c
\end{bmatrix}
\]

and the characteristic equation is given by \(\det(J - \lambda I) = 0\) where 

\[
J - \lambda I = \begin{bmatrix}
- (\nu_L + \mu_L) - \lambda & r b & 0 \\
\nu_L & -\mu_A - \lambda & 0 \\
0 & 0 & -k r_c - \lambda
\end{bmatrix}.
\]

The determinant of this matrix can be calculated by expanding along the third row, i.e.,

$$ 
\begin{aligned}
\operatorname{det}(J-\lambda I) & =(-k r_c-\lambda)  \operatorname{det}\left[\begin{array}{cc}
-\left(\nu_L+\mu_L\right)-\lambda & r b \\
\nu_L & -\mu_A-\lambda
\end{array}\right] \\
& =(-k r_c-\lambda)\left(-\left(\nu_L+\mu_L\right)-\lambda\right)\left(-\mu_A-\lambda\right)-(r b)\left(\nu_L\right) \\
& =(-k r_c-\lambda)\left(\lambda^2+\left(\left(\nu_L+\mu_L\right)+\mu_A\right) \lambda+\left(\nu_L+\mu_L\right) \mu_A-r b \nu_L\right).
\end{aligned}
$$

The eigenvalues are \(\lambda_1 = -k r_c\) and 

$$
\lambda_{2,3} = \frac{-\left(\nu_L + \mu_L + \mu_A\right) \pm \sqrt{\left(\nu_L + \mu_L + \mu_A\right)^2 - 4\left((\nu_L + \mu_L) \mu_A - r b \nu_L\right)}}{2}.
$$
Since $k>0$, the eigenvalue \(\lambda_1\) will be negative if and only if \(-k r_c < 0 \Rightarrow r_c > 0\). Moreover, since  $(\nu_L + \mu_L + \mu_A) > 0 $,  the other two eigenvalues \(\lambda_{2,3}\) will have negative real parts if and only if
$$
\begin{aligned}
& \left(\nu_L+\mu_L\right) \mu_A-r b \nu_L>0 
& \Leftrightarrow\quad {N}<1.
\end{aligned}
$$
Since $r_c > 0$ is also our model assumption, the analysis confirms that $E_{01}$ is LAS, establishing the stability conditions stated in item 2. On the other hand, $E_{01}$ loses stability if $N>1$.

 \item \underline{$E_{02}$:}  We evaluate the Jacobian matrix at $E_{02}$ and analyze the associated eigenvalues. Then we determine the conditions under which the steady state is LAS. The Jacobian matrix at $E_{02}$ is given by

\[
J = \begin{bmatrix}
-(\nu_L + \mu_L) & r b & 0 \\
\nu_L & -\mu_A & 0 \\
0 & 0 & k r_c
\end{bmatrix}.
\]

Now the characteristic equation is given by \(\det(J - \lambda I) = 0\) where 

\[
J - \lambda I = \begin{bmatrix}
-(\nu_L + \mu_L) - \lambda & r b & 0 \\
\nu_L & -\mu_A - \lambda & 0 \\
0 & 0 & k r_c - \lambda
\end{bmatrix}.
\]

The determinant of this matrix can be calculated by expanding along the third row, i.e., 

$$
\begin{aligned}
\operatorname{det}(J-\lambda I) & =\left(k r_c-\lambda\right) \operatorname{det}\left[\begin{array}{cc}
-\left(\nu_L+\mu_L\right)-\lambda & r b \\
\nu_L & -\mu_A-\lambda
\end{array}\right] \\
& =\left(k r_c-\lambda\right)\left(\lambda^2+\left[\left(\nu_L+\mu_L\right)+\mu_A\right] \lambda+\left[\left(\nu_L+\mu_L\right) \mu_A-r b \nu_L\right]\right).
\end{aligned}
$$
The first eigenvalue \(\lambda_1 = k r_c\) and the other two eigenvalues 
$$
\lambda_{2,3} = \frac{-\left((\nu_L + \mu_L) + \mu_A\right) \pm \sqrt{\left((\nu_L + \mu_L) + \mu_A\right)^2 - 4\left((\nu_L + \mu_L) \mu_A - r b \nu_L\right)}}{2}.$$
 
Since $k>0$, so the eigenvalue \(\lambda_1\) will be negative if and only if \(k r_c < 0 \Leftrightarrow  r_c < 0
\) and \(\lambda_{2,3}\) will have negative real parts if and only if $$
\left(\left(\nu_L+\mu_L\right) \mu_A-r b \nu_L\right)>0 
 \Leftrightarrow \quad {N}<1.
$$
Here $r_c < 0$ is not feasible since the perceived risk is assumed to be a positive real number. Hence $E_{02}$ is always unstable.

 \item \underline{$E_{03}$:} Next, the Jacobian matrix at $E_{03}$ is given by

$$
J=\left[\begin{array}{ccc}
-\left(\frac{r b \nu_L}{\mu_A}\left(1-\frac{1}{{N}}\right)+\nu_L+\mu_L\right) & \frac{r b}{{N}} & -r b \frac{\nu_L}{\mu_A}\left(1-\frac{1}{{N}}\right)^2\left(K_{v_{\max }}-K_{v_{\text {min }}}\right) \\
\nu_L & -\mu_A & 0 \\
0 & 0 & k\left(-r_c+r_d m\left(\frac{\nu_L}{\mu_A} K_{v_{\text {max }}}\left(1-\frac{1}{{N}}\right)\right)\right.
\end{array}\right]
$$

and the characteristic equation is given by \(\det(J - \lambda I) = 0\) where

\[
%\small{ 
J-\lambda I = \left[
\begin{array}{ccc}
-\left(\frac{r b \nu_L}{\mu_A}\left(1 - \frac{1}{{N}}\right) + \nu_L + \mu_L\right) - \lambda 
& \frac{r b}{{N}} 
& -r b \frac{\nu_L}{\mu_A}\left(1 - \frac{1}{{N}}\right)^2\left(K_{v_{\max}} - K_{v_{\min}}\right) \\
\nu_L 
& -\mu_A - \lambda 
& 0 \\
0 
& 0 
& k\left(-r_c + r_d  m \left(\frac{\nu_L}{\mu_A} K_{v_{\max}}\left(1 - \frac{1}{{N}}\right)\right)\right) - \lambda
\end{array}
\right].
%}
\]

Now the determinant of this matrix can be calculated by expanding along the third row, i.e.,

$$
\begin{aligned}
& \operatorname{det}(J-\lambda I)=\left[k\left(-r_c+r_d m\left(\frac{\nu_L}{\mu_A} K_{v_{\max }}\left(1-\frac{1}{{N}}\right)\right)\right)-\lambda\right] \\
& \operatorname{det}\left[\begin{array}{cc}
-\left(\frac{r b \nu_L}{\mu_A}\left(1-\frac{1}{{N}}\right)+\nu_L+\mu_L\right)-\lambda & \frac{r b}{{N}} \\
\nu_L & -\mu_A-\lambda
\end{array}\right]
\end{aligned}
$$

$$
\begin{aligned}
&=\left[k\left(-r_c+r_d m\left(\frac{\nu_L}{\mu_A} K_{v_{\max }}\left(1-\frac{1}{{N}}\right)\right)\right)-\lambda\right]\\
& {\left[\left(\frac{r b \nu_L}{\mu_A}\left(1-\frac{1}{{N}}\right)+\nu_L+\mu_L\right)+\lambda\right]\left[\mu_A+\lambda\right]-\left[\frac{r b}{{N}}\right]\left[\nu_L\right] }
\end{aligned}
$$

$$
\begin{aligned}
= & {\left[k\left(-r_c+r_d m\left(\frac{\nu_L}{\mu_A} K_{v_{\max }}\left(1-\frac{1}{{N}}\right)\right)\right)-\lambda\right] } \\
& {\left[\left(\frac{r b \nu_L}{\mu_A}\left(1-\frac{1}{{N}}\right)+\nu_L+\mu_L\right) \mu_A+\left(\frac{r b \nu_L}{\mu_A}\left(1-\frac{1}{{N}}\right)+\nu_L+\mu_L\right) \lambda+\mu_A \lambda+\lambda^2-\frac{r b \nu_L}{{N}}\right] }
\end{aligned}
$$

$$
\begin{aligned} 
& =\left[k\left(-r_c+r_d  m\left(\frac{\nu_L}{\mu_A} K_{v_{\max }}\left(1-\frac{1}{{N}}\right)\right)\right)-\lambda\right]  \\
& {\left[\lambda^2+\left(\left(\frac{r b \nu_L}{\mu_A}\left(1-\frac{1}{{N}}\right)+\nu_L+\mu_L\right)+\mu_A\right) \lambda+\left(\left(\frac{r b \nu_L}{\mu_A}\left(1-\frac{1}{{N}}\right)+\nu_L+\mu_L\right) \mu_A-\frac{r b \nu_L}{{N}}\right)\right]}.
\end{aligned}
$$

Hence, we have eigenvalues $\lambda_1=k\left(-r_c+r_d m\left(\frac{\nu_L}{\mu_A} K_{v_{\max }}\left(1-\frac{1}{{N}}\right)\right)\right)
$ which is true if $N>1$  and

$
\begin{aligned}
& \lambda_{2,3}=\frac{-\left(\left(\frac{r b \nu_L}{\mu_A}\left(1-\frac{1}{{N}}\right)+\nu_L+\mu_L\right)+\mu_A\right)}{2} \pm \\
& \frac{\sqrt{\left(\left(\frac{r b \nu_L}{\mu_A}\left(1-\frac{1}{{N}}\right)+\nu_L+\mu_L\right)+\mu_A\right)^2-4\left(\left(\frac{r b \nu_L}{\mu_A}\left(1-\frac{1}{{N}}\right)+\nu_L+\mu_L\right) \mu_A-\frac{r b \nu_L}{{N}}\right)}}{2}.
\end{aligned}
$

Here, eigenvalue \(\lambda_1\) will be negative if and only if $$
\begin{aligned}
& k\left[-r_c+r_d m\left(\frac{\nu_L}{\mu_A} K_{v_{\max }}\left(1-\frac{1}{{N}}\right)\right)\right]<0 \\
& \Rightarrow \frac{r_c}{r_d}>\frac{m \nu_L}{\mu_A} K_{v_{\max }}\left(1-\frac{1}{{N}}\right) \\
& \Rightarrow \frac{r_c}{r_d}>\alpha K_{v_{\max }}\left(1-\frac{1}{{N}}\right) \text { where } \alpha=\frac{m \nu_L}{\mu_A}
\end{aligned}
$$

and the real parts of the eigenvalues \(\lambda_{2, 3}\) will be negative if and only if
$$
\left(\frac{r b \nu_L}{\mu_A}\left(1-\frac{1}{{N}}\right)+\nu_L+\mu_L\right) \mu_A-\frac{r b \nu_L}{{N}}>0 
 \Leftrightarrow \quad {N}>1.
$$ 
 Since $N>1$ is also an existence condition for $E_{03}$, the analysis confirms that $E_{03}$ is LAS, establishing the stability conditions stated in item 3. On the other hand, $E_{03}$ loses stability when $\frac{r_c}{r_d}<\alpha K_{v_{\max }}\left(1-\frac{1}{N}\right)$ where $\alpha=\frac{m \nu_L}{\mu_A}$.

\item \underline{$E_{04}$:} The Jacobian matrix at $E_{04}$ is given by 
 $$
J=\left[\begin{array}{ccc}
-\left(\frac{r b \nu_L}{\mu_A}\left(1-\frac{1}{{N}}\right)+\nu_L+\mu_L\right) & \frac{r b}{{N}} & -r b \frac{\nu_L}{\mu_A}\left(1-\frac{1}{{N}}\right)^2\left(K_{v_{\max }}-K_{v_{\text {min }}}\right) \\
\nu_L & -\mu_A & 0 \\
0 & 0 & -k\left(-r_c+r_d m\left(\frac{\nu_L}{\mu_A} K_{v_{\text {min }}}\left(1-\frac{1}{{N}}\right)\right)\right.
\end{array}\right].
$$

and the characteristic equation is given by \(\det(J - \lambda I) = 0\) where

\[
%\small{
J-\lambda I = \left[
\begin{array}{ccc}
-\left(\frac{r b \nu_L}{\mu_A}\left(1 - \frac{1}{{N}}\right) + \nu_L + \mu_L\right) - \lambda 
& \frac{r b}{{N}} 
& -r b \frac{\nu_L}{\mu_A}\left(1 - \frac{1}{{N}}\right)^2\left(K_{v_{\max}} - K_{v_{\min}}\right) \\
\nu_L 
& -\mu_A - \lambda 
& 0 \\
0 
& 0 
& -k\left(-r_c + r_d m \left(\frac{\nu_L}{\mu_A} K_{v_{\min}}\left(1 - \frac{1}{{N}}\right)\right)\right) - \lambda
\end{array}
\right].
%}
\]

The determinant of this matrix can be calculated by expanding along the third row, i.e.,

$$
\begin{aligned}
& \operatorname{det}(J-\lambda I)=\left[-k\left(-r_c+r_d m\left(\frac{\nu_L}{\mu_A} K_{v_{\min }}\left(1-\frac{1}{{N}}\right)\right)\right)-\lambda\right] \\
& \operatorname{det}\left[\begin{array}{cc}
-\left(\frac{r b \nu_L}{\mu_A}\left(1-\frac{1}{{N}}\right)+\nu_L+\mu_L\right)-\lambda & \frac{r b}{{N}} \\
\nu_L & -\mu_A-\lambda
\end{array}\right]
\end{aligned}
$$

$$ 
\begin{aligned}
&=\left[-k\left(-r_c+r_d m\left(\frac{\nu_L}{\mu_A} K_{v_{\min }}\left(1-\frac{1}{{N}}\right)\right)\right)-\lambda\right]. \\
& {\left[\left(\frac{r b \nu_L}{\mu_A}\left(1-\frac{1}{{N}}\right)+\nu_L+\mu_L\right)+\lambda\right]\left[\mu_A+\lambda\right]-\left[\frac{r b}{{N}}\right]\left[\nu_L\right] }
\end{aligned}
$$

$$
\begin{aligned}
= & {\left[-k\left(-r_c+r_d m\left(\frac{\nu_L}{\mu_A} K_{v_{\min }}\left(1-\frac{1}{{N}}\right)\right)\right)-\lambda\right] } \\
& {\left[\left(\frac{r b \nu_L}{\mu_A}\left(1-\frac{1}{{N}}\right)+\nu_L+\mu_L\right) \mu_A+\left(\frac{r b \nu_L}{\mu_A}\left(1-\frac{1}{{N}}\right)+\nu_L+\mu_L\right) \lambda+\mu_A \lambda+\lambda^2-\frac{r b \nu_L}{{N}}\right] }
\end{aligned}
$$

$$
\begin{aligned}
& =\left[-k\left(-r_c+r_d t m\left(\frac{\nu_L}{\mu_A} K_{v_{\min }}\left(1-\frac{1}{{N}}\right)\right)\right)-\lambda\right] \\
& {\left[\lambda^2+\left(\left(\frac{r b \nu_L}{\mu_A}\left(1-\frac{1}{{N}}\right)+\nu_L+\mu_L\right)+\mu_A\right) \lambda+\left(\left(\frac{r b \nu_L}{\mu_A}\left(1-\frac{1}{{N}}\right)+\nu_L+\mu_L\right) \mu_A-\frac{r b \nu_L}{{N}}\right)\right]}.
\end{aligned}
$$

Here, eigenvalues: $
\lambda_1=-k\left(-r_c+r_d m\left(\frac{\nu_L}{\mu_A} K_{v_{\min }}\left(1-\frac{1}{{N}}\right)\right)\right) 
$ which is true if $N>1$ and 

$$
\begin{aligned}
& \lambda_{2,3}=\frac{-\left(\left(\frac{r b \nu_L}{\mu_A}\left(1-\frac{1}{{N}}\right)+\nu_L+\mu_L\right)+\mu_A\right)}{2} \pm \\
& \frac{\sqrt{\left(\left(\frac{r b \nu_L}{\mu_A}\left(1-\frac{1}{{N}}\right)+\nu_L+\mu_L\right)+\mu_A\right)^2-4\left(\left(\frac{r b \nu_L}{\mu_A}\left(1-\frac{1}{{N}}\right)+\nu_L+\mu_L\right) \mu_A-\frac{r b \nu_L}{{N}}\right)}}{2}.
\end{aligned}
$$

For the steady state to be LAS, the eigenvalue \(\lambda_1\) will be negative if and only if
 $$
 -k\left[-r_c+r_d m\left(\frac{\nu_L}{\mu_A} K_{v_{\min }}\left(1-\frac{1}{{N}}\right)\right)\right]<0 \\
 \Leftrightarrow \frac{r_c}{r_d}<\alpha K_{v_{\min }}\left(1-\frac{1}{{N}}\right) \text {where} \quad \alpha=\frac{m\nu_L}{\mu_A}
$$
and the eigenvalues \(\lambda_{2, 3}\) will have negative real parts if and only if
$$
\left(\left(\frac{r b \nu_L}{\mu_A}\left(1-\frac{1}{{N}}\right)+\nu_L+\mu_L\right) \mu_A-\frac{r b \nu_L}{{N}}\right)>0 
\Leftrightarrow \quad {N}>1.
$$ 
 Since $N>1$ is also an existence condition for $E_{04}$, the analysis confirms that $E_{04}$ is LAS, establishing the stability conditions stated in item 4.
Conversely, $E_{04}$ will be unstable if $\frac{r_c}{r_d}>\alpha K_{v_{\min }}\left(1-\frac{1}{N}\right)$ where $\alpha=\frac{m \nu_L}{\mu_A}$.
\end{itemize}
\end{proof}

\begin{proof}[\textbf{Proof of Theorem \ref{thm: Game-theoretic Model_}}]

We will evaluate the Jacobian matrix at the given steady state

$$E_{05} =
\left(L_v, \mathrm{~A}_v, w\right)=\left(\frac{\mu_A r_c}{r_d m \nu_L}, \frac{r_c}{r_d m}, \frac{K_{v_{\max }}-\frac{\mu_A  r_c}{r_d m \nu_L\left(1-\frac{1}{N}\right)}}{\left(K_{v_{\max }}-K_{v_{\min }}\right)}\right)
$$

where $N>1$ and \textcolor{black}{$\frac{r_c}{r_d} \in\left[\alpha K_{v_{\min }}\left(1-\frac{1}{N}\right), \alpha K_{v_{\max }}\left(1-\frac{1}{N}\right)\right]$} where $\alpha=\frac{m \nu_L}{\mu_A}.$

Given Jacobian Matrix $
J=\left[\begin{array}{ccc}
-\frac{r b A_v}{K_v(w)}-\left(\nu_L+\mu_L\right) & r b\left(1-\frac{L_v}{K_v(w)}\right) & \frac{-r b A_v L_v\left(K_{v_{\max }}-K_{v_{\min }}\right)}{\left(K_v(w)\right)^2} \\
\nu_L & -\mu_A & 0 \\
0 & k w(1-w) r_d m & k\left(-r_c+r_d m A_v\right)(1-2 w)
\end{array}\right].
$

By substituting

$$
L_v  =\frac{\mu_A r_c}{r_d m \nu_L}, \quad  
A_v  =\frac{r_c}{r_d m},\quad  \text{and} \ 
w  =\frac{K_{v_{\max }}-\frac{\mu_A r_c}{r_d m \nu_L\left(1-\frac{1}{N}\right)}}{\left(K_{v_{\max }}-K_{v_{\min }}\right)},
$$
we obtain the following entries of the matrix $J$:

\begin{itemize}

\item $J_{11}:$
$$
\begin{aligned}
& J_{11}=-\frac{r b A_v}{K_v(w)}-\left(\nu_L+\mu_L\right) \\
& =-\frac{rb\frac{r_c}{r_d m}}{\frac{\mu_A r_c}{r_d  m \nu_L\left(1-\frac{1}{N}\right)}}-\left(\nu_L+\mu_L\right) \\
& =\frac{-r b\nu_L\left(1-\frac{1}{N}\right)}{\mu_A}-\left(\nu_L+\mu_L\right)
\end{aligned}
$$

\item $J_{12}:$
$$
\begin{aligned}
J_{12} & =r b \left(1-\frac{L_v}{K_v(w)}\right) \\
& =r b\left(1-\frac{\frac{\mu_A r_c}{r_d m \nu_L}}{\frac{\mu_A r_c}{r_d m \nu_L\left(1-\frac{1}{N}\right)}}\right) \\
& =\frac{rb}{N}
\end{aligned}
$$

\item $J_{13}:$
$$
\begin{aligned}
& J_{13}=\frac{-r b A_v L_v\left(K_{v_{\max }}-K_{v_{\min }}\right)}{\left(K_v(w)\right)^2} \\
& =\frac{-r b \frac{r_c}{r_d m} \frac{\mu_A r_c}{r_d m \nu_L}\left(K_{v_{\max }}-K_{v_{\min }}\right)}{\left(\frac{\mu_A r_c}{r_d m \nu_L\left(1-\frac{1}{N}\right)}\right)^2} \\
& =\frac{-r b \nu_L \left(1-\frac{1}{N}\right)^2\left(K_{v_{\max }}-K_{v_{\min }}\right)}{\mu_A}
\end{aligned}
$$

\item  $
J_{21}=v_L, \quad J_{22}=-\mu_A, \quad J_{23}=0, \quad J_{31}=0
$

\item $J_{32}:$
$$
\begin{aligned}
& J_{32}=k w (1-w) r_d m \\
& =K \left[\frac{k_{v_{\max }}-\frac{\mu_A r_c}{r_d m \nu_L\left(1-\frac{1}{N}\right)}}{k_{v_{\max }}-k_{v_{\min }}}\right] \left[1-\left\{\frac{k_{v_{\max }}-\frac{\mu_A r_c}{r_d m \nu_L\left(1-\frac{1}{N}\right)}}{k_{v_{\max }}-k_{v_{\min }}}\right\}\right] r_d m \\
& =k r_d m\left[\frac{k_{v_{\max }}-\frac{\mu_A r_c}{r_d m \nu_L\left(1-\frac{1}{N}\right)}}{k_{v_{\max }}-k_{v_{\min }}}\right]\left[\frac{-k_{v_{\min }}+\frac{\mu_A r_c}{r_d m \nu_L\left(1-\frac{1}{N}\right)}}{k_{v_{\max }}-k_{v_{\min }}}\right] \\
& =\frac{k r_d m}{\left(k_{v_{\max }}-k_{v_{\min }}\right)^2}\left[k_{v_{\max }}-\frac{\mu_A r_c}{r_d m \nu_L\left(1-\frac{1}{N}\right)}\right]\left[-k_{v_{\min }}+\frac{\mu_A r_c}{r_d m \nu_L\left(1-\frac{1}{N}\right)}\right]
\end{aligned}
$$

$$
\begin{aligned}
&\begin{aligned}
& =\frac{k r_d m}{\left(k_{v_{\max }}-k_{v_{\min }}\right)^2}\left[\frac{\left(k_{v_{\max }} r_d m \nu_L\left(1-\frac{1}{N}\right)-\mu_A r_c\right)\left(-k_{v_{\min }} r_d m \nu_L\left(1-\frac{1}{N}\right)+\mu_A r_c\right)}{\left(r_d m \nu_L\left(1-\frac{1}{N}\right)\right)^2}\right] \\
& =\frac{k r_d m}{\left(k_{v_{\max }}-k_{v_{\min }}\right)^2}\left[\frac{r_d \mu_A\left(k_{v_{\max }}\frac{m\nu_L}{\mu_A}\left(1-\frac{1}{N}\right)-\frac{r_c}{r_d}\right) r_d \mu_A\left(-k_{v_{\min }} \frac{m \nu_L}{\mu_A}\left(1-\frac{1}{N}\right)+\frac{r_c}{r_d}\right)}{\left(r_d m \nu_L\left(1-\frac{1}{N}\right)\right)^2}\right] \\
& \left.\left.=\frac{k r_d \mu_A^2}{m  \nu_L^2\left(1-\frac{1}{N}\right)^2\left(k_{v_{\max }}-k_{v_{\min }}\right)^2}\right.\right.\left[\left(\alpha \left(1-\frac{1}{N}\right) k_{v_{\max }}-\frac{r_c}{r_d}\right)\left(-\alpha \left(1-\frac{1}{N}\right) k_{v_{\min }}+\frac{r_c}{r_d}\right)\right]
\end{aligned}\\
\end{aligned}
$$

$\text{where}, \quad \frac{m {\nu}_L}{{\mu}_A} = {\alpha}.$

\item $J_{33}:$
$$
\begin{aligned}
J_{33} & =k\left(-r_c+r_d m A_v\right)(1-2 w) \\
& =k\left(-r_c+r_d m \frac{r_c}{r_d m}\right)(1-2 w)=0.
\end{aligned}
$$
\end{itemize}
Hence the Jacobian matrix  at $$
\left(L_v, \mathrm{~A}_v, w\right)=\left(\frac{\mu_A r_c}{r_d  m \nu_L}, \frac{r_c}{r_d m}, \frac{K_{v_{\max }}-\frac{\mu_A r_c}{r_d m \nu_L\left(1-\frac{1}{N}\right)}}{\left(K_{v_{\max }}-K_{v_{\min }}\right)}\right)
$$

where $N>1$ and \textcolor{black}{$\frac{r_c}{r_d} \in\left[\alpha K_{v_{\min }}\left(1-\frac{1}{N}\right), \alpha K_{v_{\max }}\left(1-\frac{1}{N}\right)\right]$}, $\alpha=\frac{m \nu_L}{\mu_A}$ is

$$
J=\left[\begin{array}{ccc}
\frac{-r b \nu_L \left(1-\frac{1}{N}\right)}{\mu_A}-\left(v_L+\mu_L\right) & \frac{r b}{N} & \frac{-r b \nu_L \left(1-\frac{1}{N}\right)^2\left(K_{v_{\max }}-K_{v_{\min }}\right)}{\mu_A} \\
\nu_L & -\mu_A & 0 \\
0 & J_{32} & 0
\end{array}\right]
$$

where, $$
J_{32}=\frac{k r_d \mu_A^2 \left[\left(\alpha \left(1-\frac{1}{N}\right) k_{v_{\max }}-\frac{r_c}{r_d}\right)\left(-\alpha \left(1-\frac{1}{N}\right) k_{v_{\min }}+\frac{r_c}{r_d}\right)\right]}{m \nu_L^2\left(1-\frac{1}{N}\right)^2\left(k_{v_{\max }}-k_{v_{\min }}\right)^2}.
$$

The characteristic equation is given by \(\det(J - \lambda I) = 0\), i.e.,

$$
J-\lambda I=\left[\begin{array}{ccc}
\frac{-r b \nu_L \left(1-\frac{1}{N}\right)}{\mu_A}-\left(\nu_L+\mu_L\right)-\lambda & \frac{r b}{N} & \frac{-r  b \nu_L \left(1-\frac{1}{N}\right)^2\left(K_{v_{\max }}-K_{v_{\min }}\right)}{\mu_A} \\
\nu_L & -\mu_A-\lambda & 0 \\
0 & J_{32} & -\lambda
\end{array}\right]
$$ 

and the determinant of this matrix can be calculated by expanding along the third row:

$$
\operatorname{det}(J-\lambda I)=-J_{32} \operatorname{det}\left(\left[\begin{array}{cc}
-\left(\frac{r b \nu_L}{\mu_A}\left(1-\frac{1}{{N}}\right)+\nu_L+\mu_L\right)-\lambda & \frac{-r  b \nu_L \left(1-\frac{1}{N}\right)^2\left(K_{v_{\max }}-K_{v_{\min }}\right)}{\mu_A} \\
\nu_L & 0
\end{array}\right]\right) 
$$

$$
-\lambda \operatorname{det}\left(\left[\begin{array}{cc}
-\left(\frac{r b \nu_L}{\mu_A}\left(1-\frac{1}{{N}}\right)+\nu_L+\mu_L\right)-\lambda & \frac{r b}{{N}} \\
\nu_L & -\mu_A-\lambda
\end{array}\right]\right)
$$

$$
\begin{array}{r}
=-\lambda \left[\left(\frac{r b \nu_L \left(1-\frac{1}{N}\right)}{\mu_A}+\left(\nu_L+\mu_L\right)+\lambda\right)\left(\mu_A+\lambda\right)-\frac{r b \nu_L}{N}\right]  \\
-J_{32} \left[\frac{r b \nu_L^2 \left(1-\frac{1}{N}\right)^2\left(k_{v_{\max }}-k_{v_{\min }}\right)}{\mu_A}\right]
\end{array}
$$

$$
=-\lambda \left[\lambda^2+\left(\frac{r b \nu_L \left(1-\frac{1}{N}\right)}{\mu_A}+\left(\nu_L+\mu_L+\mu_A\right)\right) \lambda+\left(r  b \nu_L \left(1-\frac{2}{N}\right)+\left(\nu_L+\mu_L\right) \mu_A\right)\right]
$$

$$
\begin{aligned}
-\frac{k r_d \mu_A^2}{m \nu_L^2\left(1-\frac{1}{N}\right)^2\left(k_{v_{\max }}-k_{v_{\min }}\right)^2} \frac{r b \nu_L^2 \left(1-\frac{1}{N}\right)^2\left(k_{v_{\max }}-k_{v_{\min }}\right)}{\mu_A} \\
\left[\left(\alpha \left(1-\frac{1}{N}\right) k_{v_{\max }}-\frac{r_c}{r_d}\right)\left(-\alpha \left(1-\frac{1}{N}\right) k_{v_{\min }}+\frac{r_c}{r_d}\right)\right]
\end{aligned}
$$

$$
=-\lambda \left[\lambda^2+\left(\frac{r b \nu_L \left(1-\frac{1}{N}\right)}{\mu_A}+\left(\nu_L+\mu_L+\mu_A\right)\right) \lambda+r b \nu_L \left(\left(1-\frac{2}{N}\right)+\frac{\left(\nu_L+\mu_L\right) \mu_A}{r b \nu_L}\right)\right]
$$

$$
-\frac{k r b r_d \mu_A}{m \left(k_{v_{\max }}-k_{v_{\min }}\right)} \left[\left(\alpha \left(1-\frac{1}{N}\right) k_{v_{\max }}-\frac{r_c}{r_d}\right)\left(-\alpha \left(1-\frac{1}{N}\right) k_{v_{\min }}+\frac{r_c}{r_d}\right)\right]
$$

$$
\begin{aligned}
& =-\lambda^3-\left(\frac{r b \nu_L \left(1-\frac{1}{N}\right)}{\mu_A}+\left(\nu_L+\mu_L+\mu_A\right)\right) \lambda^2-\left(r b \nu_L \left(1-\frac{2}{N}\right)+\frac{1}{N}\right) \lambda \\
& -P \left[\left(\alpha \left(1-\frac{1}{N}\right) k_{v_{\max }}-\frac{r_c}{r_d}\right)\left(-\alpha \left(1-\frac{1}{N}\right) k_{v_{\min }}+\frac{r_c}{r_d}\right)\right]
\end{aligned}
$$

$$
\begin{aligned}
& =-\lambda^3-\left(\frac{r b \nu_L \left(1-\frac{1}{N}\right)}{\mu_A}+\left(\nu_L+\mu_L+\mu_A\right)\right) \lambda^2-\left(r b \nu_L \left(1-\frac{1}{N}\right)\right) \lambda \\
& -P \left[\left(\alpha \left(1-\frac{1}{N}\right) k_{v_{\max }}-\frac{r_c}{r_d}\right)\left(-\alpha \left(1-\frac{1}{N}\right) k_{v_{\min }}+\frac{r_c}{r_d}\right)\right].
\end{aligned}
$$

The characteristic equation is given by $\operatorname{det}(J-\lambda I)=0$, i.e.,

$$
\begin{aligned}
& \lambda^3+\left(\frac{r b \nu_L \left(1-\frac{1}{N}\right)}{\mu_A}+\left(\nu_L+\mu_L+\mu_A\right)\right) \lambda^2+\left(r b \nu_L \left(1-\frac{1}{N}\right)\right) \lambda \\
& +P \left[\left(\alpha \left(1-\frac{1}{N}\right) k_{v_{\max }}-\frac{r_c}{r_d}\right)\left(-\alpha \left(1-\frac{1}{N}\right) k_{v_{\min }}+\frac{r_c}{r_d}\right)\right]=0
\end{aligned}
$$

$$
\text { where, } P=\frac{k r b r_d \mu_A}{m \left(k_{v_{\max }}-k_{v_{\min }}\right)}>0, \quad \alpha=\frac{m \nu_L}{\mu_A}>0, \quad N=\frac{r b \nu_L}{\left(\nu_L+\mu_L\right) \mu_A}>1.
$$

Therefore, we have the cubic characteristic equation:

$$
\lambda^3+a_2 \lambda^2+a_1 \lambda+a_0=0
$$

where the coefficients $a_2, a_1$, and $a_0$ are given by

\begin{itemize}
\item $ \quad a_2=\frac{r b \nu_L \left(1-\frac{1}{N}\right)}{\mu_A}+\left(\nu_L+\mu_L+\mu_A\right)$

\item $ \quad a_1=r b \nu_L \left(1-\frac{1}{N}\right)$

\item $ \quad a_0=P \left[\left(\alpha \left(1-\frac{1}{N}\right) K_{v_{\max }}-\frac{r_c}{r_d}\right)\left(-\alpha \left(1-\frac{1}{N}\right) K_{v_{\min }}+\frac{r_c}{r_d}\right)\right]$ with $P=\frac{k r  b r_d \mu_A}{m \left(K_{v_{\max }}-K_{v_{\min }}\right)}>0$.
\end{itemize}

For the characteristic equation, the necessary and sufficient conditions (Routh-Hurwitz stability criterion \cite{ogata2009modern}) for stability are
\begin{equation*}
a_2>0,\; a_1>0,\; a_0>0 \quad \text{and} \quad a_2 a_1>a_0 .
\end{equation*}

These conditions ensure that all eigenvalues have negative real parts. We have already shown that the first three conditions are satisfied in the main text. Now we analyze the last condition as follows:

\textbf{Conditions: ($a_2 a_1 > a_0$)}

$$\begin{aligned}
& a_2 a_1 > a_0 \\[0.3cm]
& \Rightarrow \left(\frac{r b \nu_L\left(1-\frac{1}{N}\right)}{\mu_A} + \nu_L + \mu_L + \mu_A\right)\left(r b \nu_L\left(1-\frac{1}{N}\right)\right) \\
& \quad > P \left[\alpha\left(1-\frac{1}{N}\right) K_{v_{\max }} - \frac{r_c}{r_d}\right]\left[-\alpha\left(1-\frac{1}{N}\right) K_{v_{\min }} + \frac{r_c}{r_d}\right] \\[0.3cm]
& \Rightarrow \frac{r b \nu_L}{\mu_A} \left[\left(1-\frac{1}{N}\right) + \frac{(\nu_L + \mu_L) \mu_A}{r b \nu_L} + \frac{\mu_A^2}{r b \nu_L}\right]\left(r b \nu_L\left(1-\frac{1}{N}\right)\right) \\
& \quad > P \left[\alpha\left(1-\frac{1}{N}\right) K_{v_{\max }} - \frac{r_c}{r_d}\right]\left[-\alpha\left(1-\frac{1}{N}\right) K_{v_{\min }} + \frac{r_c}{r_d}\right] \\[0.3cm]
& \Rightarrow \frac{r^2 b^2 \nu_L^2}{\mu_A} \left(1 + \frac{\mu_A^2}{r b \nu_L}\right)\left(1-\frac{1}{N}\right) \\
& \quad > P \left[\alpha\left(1-\frac{1}{N}\right) K_{v_{\max }} - \frac{r_c}{r_d}\right]\left[-\alpha\left(1-\frac{1}{N}\right) K_{v_{\min }} + \frac{r_c}{r_d}\right] \\[0.3cm]
& \Rightarrow \frac{r b \nu_L\left(r b \nu_L + \mu_A^2\right)\left(1-\frac{1}{N}\right)}{\mu_A} \\
& \quad > \frac{k r b r_d  \mu_A}{m \left(K_{v_{\max }} - K_{v_{\min }}\right)}\left[\alpha\left(1-\frac{1}{N}\right) K_{v_{\max }} - \frac{r_c}{r_d}\right]\left[-\alpha\left(1-\frac{1}{N}\right) K_{v_{\min }} + \frac{r_c}{r_d}\right] \\[0.3cm]
& \Rightarrow \frac{m \nu_L \left(K_{v_{\max }} - K_{v_{\min }}\right)\left(r b \nu_L + \mu_A^2\right)\left(1-\frac{1}{N}\right)}{k \mu_A^2 r_d} \\
& \quad > \left[\alpha\left(1-\frac{1}{N}\right) K_{v_{\max }} - \frac{r_c}{r_d}\right]\left[-\alpha\left(1-\frac{1}{N}\right) K_{v_{\min }} + \frac{r_c}{r_d}\right] \\[0.3cm]
& \Rightarrow Q \alpha \left(K_{v_{\max }} - K_{v_{\min }}\right)\left(1-\frac{1}{N}\right) \\
& \quad > \left[\alpha\left(1-\frac{1}{N}\right) K_{v_{\max }} - \frac{r_c}{r_d}\right]\left[-\alpha\left(1-\frac{1}{N}\right) K_{v_{\min }} + \frac{r_c}{r_d}\right], \\[0.3cm]
& \text{where }Q=\frac{\mu_A^2+r b \nu_L}{k \mu_A r_d}, \quad P = \frac{k r b r_d \mu_A}{m \left(K_{v_{\max }} - K_{v_{\min }}\right)} > 0.
\end{aligned}
$$

Now we assume \(x = r_c / r_d\), so the above inequality becomes:
\[
Q \alpha \left(K_{v_{\max}} - K_{v_{\min}}\right) \left(1 - \frac{1}{N}\right) > \left[\alpha \left(1 - \frac{1}{N}\right) K_{v_{\max}} - x\right] \left[-\alpha \left(1 - \frac{1}{N}\right) K_{v_{\min}} + x\right]
\]

which implies:
\[
Q \alpha \left(K_{v_{\max}} - K_{v_{\min}}\right) \left(1 - \frac{1}{N}\right) > \left[-\alpha^2 \left(1 - \frac{1}{N}\right)^2 K_{v_{\max}} K_{v_{\min}} + \alpha \left(1 - \frac{1}{N}\right) (K_{v_{\max}} + K_{v_{\min}}) x - x^2 \right]
\]

Rearranging the inequality, we have:
\[
x^2 - \alpha \left(1 - \frac{1}{N}\right) (K_{v_{\max}} + K_{v_{\min}}) x + \alpha^2 \left(1 - \frac{1}{N}\right)^2 K_{v_{\max}} K_{v_{\min}} + Q \alpha \left(K_{v_{\max}} - K_{v_{\min}}\right) \left(1 - \frac{1}{N}\right) > 0
\]

which is a quadratic inequality in \(x\) (where \(x = r_c / r_d\)) and we can write it as: $f(x)=
x^2 - B x + C > 0.$

Here
\[
B = \alpha \left(1 - \frac{1}{N}\right) (K_{v_{\max}} + K_{v_{\min}}),
\]
\[
C = \alpha^2 \left(1 - \frac{1}{N}\right)^2 K_{v_{\max}} K_{v_{\min}} + Q \alpha \left(K_{v_{\max}} - K_{v_{\min}}\right) \left(1 - \frac{1}{N}\right).
\]
The roots of the quadratic equation are given by
\[
x_{1,2} = \frac{B \pm \sqrt{B^2 - 4C}}{2}
\]
i.e.,
$$
x_1=\frac{r_c}{r_d}=\alpha \frac{\left(K_{v_{\max }}+K_{v_{\min }}\right)}{2} \left(1-\frac{1}{N}\right)-\frac{\sqrt{B^2-4 C}}{2}
$$
and
$$
x_2=\frac{r_c}{r_d}=\alpha \frac{\left(K_{v_{\max }}+K_{v_{\min }}\right)}{2} \left(1-\frac{1}{N}\right)+\frac{\sqrt{B^2-4 C}}{2}.
$$

Now we consider the following two cases:

\textbf{Case I:}

If $B^2-4 C<0$ then the roots $x_1, x_2$ are complex conjugates, which implies no real instability threshold exists.

In this case for any $x$ in the interval $\alpha\left(1-\frac{1}{N}\right) K_{v_{\min }}<\frac{r_c}{r_d}<\alpha\left(1-\frac{1}{N}\right) K_{v_{\max }}$, we obtain \[
f(x)=x^2 - B x + C > 0.
\]

Hence, the co-existence steady state $E_{05}$  is stable in the whole set given by Routh-Hurwitz criterion, which is $\alpha\left(1-\frac{1}{N}\right) K_{v_{\min }}<\frac{r_c}{r_d}<\alpha\left(1-\frac{1}{N}\right) k_{v_{\max }}$ and the required condition is

$$
\begin{gathered} 
B^2-4 C<0 \\
\Rightarrow \alpha^2\left(1-\frac{1}{N}\right)^2\left(k_{v_{\max }}+k_{v_{\min }}\right)^2-4 \alpha^2\left(1-\frac{1}{N}\right)^2 k_{v_{\max }} k_{v_{\min }} \\
-4 Q\alpha\left(k_{v_{\max }}-k_{v_{\min }}\right)\left(1-\frac{1}{N}\right)<0 \\
\Rightarrow \alpha^2\left(1-\frac{1}{N}\right)^2\left(k_{v_{\max }}-k_{v_{\min }}\right)^2-4 Q\alpha\left(k_{v_{\max }}-k_{v_{\min }}\right)\left(1-\frac{1}{N}\right)<0 \\
\Rightarrow \alpha\left(k_{v_{\max }}-k_{v_{\min }}\right)\left(1-\frac{1}{N}\right)<4Q.
\end{gathered}
$$

\textbf{Case II:}

If the condition $B^2-4 C>0$ is satisfied, i.e., $\Rightarrow \alpha\left(k_{v_{\max }}-k_{v_{\min }}\right)\left(1-\frac{1}{N}\right)>4 Q$, then the quadratic equation $f(x)=x^2-B x+C=0$ has the following two roots, given by

$$
x_{1,2}=\frac{B \pm \sqrt{B^2-4 C}}{2}
$$

where

$$
\begin{gathered}
B=\alpha\left(1-\frac{1}{N}\right)\left(K_{v_{\max }}+K_{v_{\min }}\right), \\
C=\alpha^2\left(1-\frac{1}{N}\right)^2 K_{v_{\max }} K_{v_{\min }}+Q \alpha\left(K_{v_{\max }}-K_{v_{\min }}\right)\left(1-\frac{1}{N}\right) .
\end{gathered}
$$

In the following calculation, we assume

$$
\alpha\left(1-\frac{1}{N}\right) K_{v_{\max }}=\chi_{\max} \quad\text {and} \quad \alpha\left(1-\frac{1}{N}\right) K_{v_{\min }}=\chi_{\min}.
$$

Plugging these values, we obtain

$$
\begin{gathered}
x_{1,2}=\frac{\left(\chi_{\max }+\chi_{\min }\right) \pm \sqrt{\left(\chi_{\max }+\chi_{\min }\right)^2-4\left(\chi_{\max } \chi_{\min }+Q\left(\chi_{\max }-\chi_{\min }\right)\right)}}{2} \\
=\frac{\left(\chi_{\max }+\chi_{\min }\right) \pm \sqrt{\left(\chi_{\max }-\chi_{\min }\right)^2-4 Q\left(\chi_{\max }-\chi_{\min }\right)}}{2}.
\end{gathered}
$$
Now,
$$
\begin{gathered}
x_1=\frac{\left(\chi_{\max }+\chi_{\min }\right)-\sqrt{\left(\chi_{\max }-\chi_{\min }\right)^2-4 Q\left(\chi_{\max }-\chi_{\min }\right)}}{2}>\frac{\left(\chi_{\max }+\chi_{\min }\right)-\sqrt{\left(\chi_{\max }-\chi_{\min }\right)^2}}{2}=\chi_{\min } \\
\Rightarrow x_1>\chi_{\min }
\end{gathered}
$$
i.e.,
$$
x_1(N)=\frac{r_c}{r_d}(N)=\alpha \frac{\left(K_{v_{\max }}+K_{v_{\min }}\right)}{2}\left(1-\frac{1}{N}\right)-\frac{\sqrt{B^2-4 C}}{2}>\alpha\left(1-\frac{1}{N}\right) K_{v_{\min }}.
$$
Next,
$$
\begin{gathered}
x_2=\frac{\left(\chi_{\max }+\chi_{\min }\right)+\sqrt{\left(\chi_{\max }-\chi_{\min }\right)^2-4 Q\left(\chi_{\max }-\chi_{\min }\right)}}{2}<\frac{\left(\chi_{\max }+\chi_{\min }\right)+\sqrt{\left(\chi_{\max }-\chi_{\min }\right)^2}}{2}=\chi_{\max } \\
\Rightarrow x_2<\chi_{\max }
\end{gathered}
$$
i.e.,
$$
x_2(N)=\frac{r_c}{r_d}(N)=\alpha \frac{\left(K_{v_{\max }}+K_{v_{\min }}\right)}{2}\left(1-\frac{1}{N}\right)+\frac{\sqrt{B^2-4 C}}{2}<\alpha\left(1-\frac{1}{N}\right) K_{v_{\max }} .
$$
Hence, the two distinct real roots $x_1(N)$ and $x_2(N)$ of $f(x)$ always lies within the interval $\left(\alpha\left(1-\frac{1}{N}\right) K_{v_{\min }}, \alpha\left(1-\frac{1}{N}\right) K_{v_{\max }}\right)$ if $B^2-4 C>0$, i.e.,

$$
\left(x_1, x_2\right) \subset\left(\alpha\left(1-\frac{1}{N}\right) K_{v_{\min }}, \alpha\left(1-\frac{1}{N}\right) K_{v_{\max }}\right) .
$$

Therefore, $E_{05}$ will be stable in the following parameter regions, given by

$$
\alpha\left(1-\frac{1}{N}\right) k_{v_{\min }}<\frac{r_c}{r_d}<x_1 \text { and } x_2<\frac{r_c}{r_d}<\alpha\left(1-\frac{1}{N}\right) k_{v_{\max }} .
$$

Finally, by combining Routh-Hurwitz stability criterion, we have the following conditions for the local stability of $E_{05}$: 

\begin{enumerate}

\item[] (i) If $B^2-4 C<0$ i.e., $\Rightarrow \alpha\left(k_{v_{\max }}-k_{v_{\min }}\right)\left(1-\frac{1}{N}\right)<4 Q$, then $E_{05}$ will be stable in the interval: 
$\alpha\left(1-\frac{1}{N}\right) K_{v_{\min }}<\frac{r_c}{r_d}<\alpha\left(1-\frac{1}{N}\right) k_{v_{\max }}$.

\item[] (ii) If $B^2-4 C>0$ i.e., $\Rightarrow \alpha\left(k_{v_{\max }}-k_{v_{\min }}\right)\left(1-\frac{1}{N}\right)>4 Q,$ then $E_{05}$ will be stable in the intervals 

$$
\begin{aligned}
&\left[\alpha \frac{\left(K_{v_{\max }}+K_{v_{\min }}\right)}{2} \left(1-\frac{1}{N}\right)+\frac{\sqrt{B^2-4 C}}{2}\right]<\frac{r_c}{r_d}<\alpha  K_{v_{\max }} \left(1-\frac{1}{N}\right)\\
&\text { and }\\
&\alpha K_{v_{\min }} \left(1-\frac{1}{N}\right)<\frac{r_c}{r_d}<\left[\alpha \frac{\left(K_{v_{\max }}+K_{v_{\min }}\right)}{2} \left(1-\frac{1}{N}\right)-\frac{\sqrt{B^2-4 C}}{2}\right]
\end{aligned}
$$

where

$$
\alpha=\frac{m \nu_L}{\mu_A}, \quad B=\alpha\left(1-\frac{1}{N}\right)\left(K_{v_{\max }}+K_{v_{\min }}\right),
$$
$$
C=\alpha^2\left(1-\frac{1}{N}\right)^2 K_{v_{\max }} K_{v_{\min }}+Q\alpha \left(K_{v_{\max }}-K_{v_{\min }}\right)\left(1-\frac{1}{N}\right), Q=\frac{\mu_A^2+r b \nu_L}{k \mu_A r_d},
$$
$$
x_1=\frac{r_c}{r_d}=\alpha \frac{\left(K_{v_{\max }}+K_{v_{\min }}\right)}{2} \left(1-\frac{1}{N}\right)-\frac{\sqrt{B^2-4 C}}{2}
$$
and
$$
x_2=\frac{r_c}{r_d}=\alpha \frac{\left(K_{v_{\max }}+K_{v_{\min }}\right)}{2} \left(1-\frac{1}{N}\right)+\frac{\sqrt{B^2-4 C}}{2}.
$$
 \end{enumerate}   
\end{proof}

\newpage
\subsection{A model with constant payoffs and public health interventions}

 We examine the impact of persuasive public health campaigns and interventions on our model~\eqref{eq:simplified_model_1}. We assume $\gamma$ as a public health (PH) action parameter to explore how a steady influence from PH authorities affects the population dynamics. Following the approach from \cite {asfaw_itn_control}, we incorporate a new term $\gamma(1-w)$ as follows:
\begin{equation}
\label{eq:PH_extension}
\frac{dw}{dt} = kw(1 - w)(-r_c + r_d) + \gamma(1 - w) 
\end{equation}

where $\gamma>0$ is a constant parameter representing the effectiveness of steady PH actions influencing household behavior. This assumption is practical for scenarios where PH efforts are sustained and consistent, such as ongoing education and awareness campaigns or continuous resource distribution.  The rest of the model structure, including the mosquito dynamics and carrying capacity \( K_v(w) \) remains the same as described in Equations~\eqref{eq:simplified_model_1}.

\textbf{Steady states}

To determine biologically plausible (nonnegative) steady states of the extended system,  we solve for \( (L_v, A_v, w) \) such that the time derivatives in the equations for $L_v$ and $A_v$ (see main text) and Equation~\eqref{eq:PH_extension} are zero. We must solve the following system:

\[
\begin{aligned}
& rb A_v \left(1 - \frac{L_v}{K_v(w)}\right) - (\nu_L + \mu_L) L_v = 0, \\
& \nu_L L_v - \mu_A A_v = 0, \\
& kw(1 - w)(-r_c + r_d) + \gamma(1 - w) = 0,
\end{aligned}
\]
where
\[
K_v(w) = K_{v_{\max}} - w(K_{v_{\max}} - K_{v_{\min}}).
\]

We solve the system and summarize the four steady states in the following proposition. 

%\newpage

\begin{proposition}\label{prop: Model with PH Intervention}

The steady states $\left(L_v^*, A_v^*, w^*\right)$ of the extended system, along with their interpretation, are as follows:

$\underline{\tilde{E}_{01}}:\left(L_v^*, A_v^*, w^*\right)=\left(0,0, \frac{\gamma}{k\left(r_c-r_d\right)}\right)$ where $0 \leq \frac{\gamma}{k\left(r_c-r_d\right)} \leq 1$: Mosquito-free, partial breeding site control.

$\underline{E_{02}}:\left(L_v^*, A_v^*, w^*\right)=(0,0,1)$: Mosquito-free, full breeding site control.

$\underline{\tilde{E}_{03}}:\left(L_v^*, A_v^*, w^*\right)=\left(K_v{ }^*\left(1-\frac{1}{N}\right), \frac{\nu_L}{\mu_A} K_v{ }^*\left(1-\frac{1}{N}\right), \frac{\gamma}{k\left(r_c-r_d\right)}\right)$ where $0 \leq \frac{\gamma}{k\left(r_c-r_d\right)} \leq 1$ and 

$K_v^*=K_v\left(\frac{\gamma}{k\left(r_e-r_d\right)}\right)=K_{v_{\text {max }}}-\frac{\gamma}{k\left(r_e-r_d\right)}\left(K_{v_{\text {max }}}-K_{v_{\text {min }}}\right)$ if and only if $N>1$: Mosquito-positive, partial breeding site control. 

$\underline{E_{04}:}\left(L_v^*, A_v^*, w^*\right)=\left(K_{v_{\text {min }}}\left(1-\frac{1}{N}\right), \frac{\nu_L}{\mu_A} K_{v_{\text {min }}}\left(1-\frac{1}{N}\right), 1\right)$ if and only if $N>1$: Mosquito-positive, full breeding site control.

\end{proposition}

\begin{proof}

By setting the time derivative to zero, from Equation~\eqref{eq:PH_extension}, we have  $$
\begin{aligned}
& kw(1-w)(-r_c + r_d)+\gamma(1-w)=0 \\
\Rightarrow \quad & (1-w)[wk(r_d - r_c)+\gamma]=0.
\end{aligned}
$$

This equation has two solutions given by $w=1$ or $w=\frac{\gamma}{k(r_c - r_d)},$  where $0 \leq \frac{\gamma}{k\left(r_c-r_d\right)} \leq 1$. Now from the equation \[\nu_L L_v - \mu_A A_v = 0,
\] we have
\[
A_v = \frac{\nu_L}{\mu_A} L_v.
\]
By substituting \(A_v = \frac{\nu_L}{\mu_A} L_v\) and \(K_v(w) = K_{v_{\max}} - w (K_{v_{\max}} - K_{v_{\min}})\) in the following equation, we have
\[
r b \left(\frac{\nu_L}{\mu_A} L_v \right) \left(1 - \frac{L_v}{K_v(w)}\right) - (\nu_L + \mu_L) L_v = 0.
\]
This gives two solutions given by $L_v = 0$ or
\[
\left(\frac{r b \nu_L}{\mu_A}\right) \left(1 - \frac{L_v}{K_v(w)}\right) = \nu_L + \mu_L
\]
which implies
$$
L_v  = K_v(w)\left(1-\frac{\mu_A\left(\nu_L+\mu_L\right)}{r b \nu_L}\right) 
\Rightarrow L_v  =K_v(w)\left(1-\frac{1}{N}\right) \text { if and only if } N>1.
$$
Therefore, $A_v=\frac{\nu_L}{\mu_A} L_v=\frac{\nu_L}{\mu_A} K_v(w)\left(1-\frac{1}{N}\right)$ and the following holds:

\begin{itemize}

\item When $w=\frac{\gamma}{k(r_c-r_d)}, L_v=0$ then we get $A_v=0.$

\item When $w=1, L_v=0$ then we get $A_v=0.$

\item When $w=\frac{\gamma}{k(r_c-r_d)}$ where $0 \leq \frac{\gamma}{k\left(r_c-r_d\right)} \leq 1$, we get $L_v=K_v\left(\frac{\gamma}{k(r_c-r_d)}\right)\left(1-\frac{1}{N}\right)=K_v^*\left(1-\frac{1}{N}\right)$ 

and $A_v=\frac{\nu_L}{\mu_A} K_v\left(\frac{\gamma}{k(r_c-r_d)}\right)\left(1-\frac{1}{N}\right)=\frac{\nu_L}{\mu_A} K_v^*\left(1-\frac{1}{N}\right)$
where, $K_v^*=K_v\left(\frac{\gamma}{k(r_c-r_d)}\right)=K_{v_{\text {max }}}-\frac{\gamma}{k(r_c-r_d)}\left(K_{v_{\text {max}}}-K_{v_{\text {min}}}\right)$.

\item When $w=1$, we get $L_v=K_v(1)\left(1-\frac{1}{N}\right)=K_{v_{\text {min }}}\left(1-\frac{1}{N}\right)$ and
$
A_v=\frac{\nu_L}{\mu_A} K_v(1)\left(1-\frac{1}{N}\right)=\frac{\nu_L}{\mu_A} K_{v_{\min }}\left(1-\frac{1}{N}\right).
$

\end{itemize}

Therefore, we obtained the following  steady states:

\begin{enumerate}

\item $\underline{\tilde{E}_{01}}:\left(L_v^*, A_v^*, w^*\right)=\left(0,0, \frac{\gamma}{k\left(r_c-r_d\right)}\right)$ where $0 \leq \frac{\gamma}{k\left(r_c-r_d\right)} \leq 1$,

\item $\underline{E_{02}}:\left(L_v^*, A_v^*, w^*\right)=(0,0,1)$,

\item $\underline{\tilde{E}_{03}}:\left(L_v^*, A_v^*, w^*\right)=\left(K_v{ }^*\left(1-\frac{1}{N}\right), \frac{\nu_L}{\mu_A} K_v{ }^*\left(1-\frac{1}{N}\right), \frac{\gamma}{k\left(r_c-r_d\right)}\right)$ where $0 \leq \frac{\gamma}{k\left(r_c-r_d\right)} \leq 1$ and

$K_v^*=K_v\left(\frac{\gamma}{k\left(r_e-r_d\right)}\right)=K_{v_{\text {max}}}-\frac{\gamma}{k\left(r_e-r_d\right)}\left(K_{v_{\text {max}}}-K_{v_{\text {min}}}\right)$ if and only if $N>1$,

\item $\underline{E_{04}}:\left(L_v^*, A_v^*, w^*\right)=\left(K_{v_{\min }}\left(1-\frac{1}{N}\right), \frac{\nu_L}{\mu_A} K_{v_{\min }}\left(1-\frac{1}{N}\right), 1\right)$ if and only if $N>1$.

\end{enumerate}
\end{proof}

We perform a local stability analysis of the four steady states, and the stability depends on the behavioral parameters \( r_c \), \( r_d \), $\gamma$ and entomological parameter $N$. Following the approach used in section 3.1, we apply the Jacobian matrix and eigenvalue analysis to establish the stability conditions, which is the content of the following theorem:

\begin{theorem}
The local stability of the steady states of the system with public health intervention is characterized as follows:
\begin{enumerate}

\item If $r_c-r_d>\frac{\gamma}{k}$, and $N<1$, then $\tilde{E}_{01}$ is LAS; $E_{02}$ is unstable.

\item If $r_c-r_d<\frac{\gamma}{k}$, and $N<1$, then $E_{02}$ is LAS.

\item If $r_c-r_d>\frac{\gamma}{k}$, and $N>1$, then $\tilde{E}_{03}$ is LAS; $\tilde{E}_{01}, E_{02}$, and $E_{04}$ are unstable.

\item If $r_c-r_d<\frac{\gamma}{k}$, and $N>1$, then $E_{04}$ is LAS; $E_{02}$ is unstable.

\end{enumerate}
\label{thm:Model with PH Intervention1}
\end{theorem}

\begin{proof} 

We proceed to examine each steady state separately and verify the claims presented in items 1--4.

 \begin{itemize}

  \item $\underline{\tilde{E}_{01}}$: We evaluate the Jacobian matrix at the steady state ${\tilde{E}_{01}}$ and find the eigenvalues. Then we derive the conditions for stability. The Jacobian matrix at ${\tilde{E}_{01}}$ is given by

$$
J-\lambda I=\left[\begin{array}{ccc}
-\left(\nu_L+\mu_L\right)-\lambda & r b & 0 \\
\nu_L & -\mu_A-\lambda & 0 \\
0 & 0 & (k(- r_c +r_d)+\gamma)-\lambda
\end{array}\right].
$$

The determinant of this matrix can be calculated by expanding along the third row, i.e.,

$$
\begin{aligned}
& \left.\operatorname{det}(J-\lambda I)=\left(k\left(-r_c+r_d\right)+\gamma\right)-\lambda\right) \operatorname{det}\left[\begin{array}{cc}
-\left(\nu_L+\mu_L\right)-\lambda & r b \\
\nu_L & -\mu_A-\lambda
\end{array}\right] \\
& \left.=\left(k\left(-r_c+r_d\right)+\gamma-\lambda\right)\right)\left(\lambda^2+\left[\left(\nu_L+\mu_L\right)+\mu_A\right] \lambda+\left[\left(\nu_L+\mu_L\right) \mu_A-r b \nu_L\right]\right)
\end{aligned}
$$

The first eigenvalue is \(\lambda_1 = k(-r_c + r_d)+\gamma\) and the other two eigenvalues are given by 
$$\lambda_{2,3} = \frac{-\left[(\nu_L + \mu_L) + \mu_A\right] \pm \sqrt{\left[(\nu_L + \mu_L) + \mu_A\right]^2 - 4\left[(\nu_L + \mu_L) \mu_A - r b \nu_L\right]}}{2}.$$

Since $k>0$, the eigenvalue \(\lambda_1\) will be negative if and only if $k\left(-r_c+r_d\right)+\gamma<0 \Leftrightarrow r_c-r_d>\frac{\gamma}{k}$. Moreover, since  $(\nu_L + \mu_L) + \mu_A > 0 $,  the other two eigenvalues \(\lambda_{2,3}\) will have negative real parts if and only if
$$
\begin{aligned}
& \left(\nu_L+\mu_L\right) \mu_A-r b \nu_L>0 
& \Leftrightarrow\quad {N}<1.
\end{aligned}
$$

 Since $r_c-r_d>\frac{\gamma}{k}$ is also an existence condition for ${\tilde{E}_{01}}$, the analysis confirms that ${\tilde{E}_{01}}$ is LAS, establishing the stability conditions stated in item 1, and becomes unstable when $N>1$.

 \item \underline{$E_{02}$:} Next, we analyze the local stability of the steady state $E_{02}$. The characteristic equation is given by \(\det(J - \lambda I) = 0\), where

$$
J-\lambda I =\left[\begin{array}{ccc}
-\left(\nu_L+\mu_L\right)-\lambda & r b & 0 \\
\nu_L & -\mu_A-\lambda & 0 \\
0 & 0 & (k(r_c - r_d) -\gamma)-\lambda
\end{array}\right]. 
$$

The determinant of this matrix can be calculated by expanding along the third row, i.e., 

$$
\begin{aligned}
\operatorname{det}(J & -\lambda I)=((k(r_c - r_d)-\gamma)-\lambda) \operatorname{det}\left[\begin{array}{cc}
-\left(\nu_L+\mu_L\right)-\lambda & r b \\
\nu_L & -\mu_A-\lambda
\end{array}\right] \\
& =((k(r_c - r_d) -\gamma)-\lambda) \left(\lambda^2+\left[\left(\nu_L+\mu_L\right)+\mu_A\right] \lambda+\left[\left(\nu_L+\mu_L\right) \mu_A-r b \nu_L\right]\right).
\end{aligned}
$$

The characteristic equation is
$
\begin{aligned}
&((k(r_c - r_d)-\gamma)-\lambda) \left(\lambda^2+\left[\left(\nu_L+\mu_L\right)+\mu_A\right] \lambda+\left[\left(\nu_L+\mu_L\right) \mu_A-r b \nu_L\right]\right)=0.
\end{aligned}
$

A first eigenvalue is thus given by $\lambda_1=k(r_c - r_d) -\gamma$
and the other two are roots of the quadratic polynomial, i.e., 
\[
\lambda_{2,3} = \frac{-\left[(\nu_L + \mu_L) + \mu_A\right] \pm \sqrt{\left[(\nu_L + \mu_L) + \mu_A\right]^2 - 4\left[(\nu_L + \mu_L) \mu_A - r b \nu_L\right]}}{2}.
\]

The eigenvalue \(\lambda_1\) will be negative if
$k\left(r_c-r_d\right)-\gamma<0 \Leftrightarrow r_c-r_d<\frac{\gamma}{k}$
and  the other two eigenvalues \(\lambda_{2,3}\) will have negative real parts if $$
\begin{aligned}
& \left(\nu_L+\mu_L\right) \mu_A-r b \nu_L>0 
& \Leftrightarrow\quad {N}<1.
\end{aligned}$$

Therefore, $E_{02}$ is LAS, confirming the stability conditions stated in item 2. Conversely, $E_{02}$ becomes unstable if $r_c-r_d>\frac{\gamma}{k}$, or $N>1$, since these conditions lead to at least one positive eigenvalue.

\item $\underline{\tilde{E}_{03}}$: The Jacobian matrix at ${\tilde{E}_{03}}$ is given by

$$
J=\left[\begin{array}{ccc}
-\left(\frac{r b \nu_L}{\mu_A}\left(1-\frac{1}{N}\right)+\nu_L+\mu_L\right) & \frac{r b}{N} & -r b \frac{\nu_L}{\mu_A}\left(1-\frac{1}{{N}}\right)^2\left(K_{v_{\max }}-K_{v_{\min }}\right) \\
\nu_L & -\mu_A & 0 \\
0 & 0 & k(- r_c+ r_d)+\gamma
\end{array}\right]
$$

and the characteristic equation is given by \(\det(J - \lambda I) = 0\) where

$$
J-\lambda I =\left[\begin{array}{ccc}
-\left(\frac{r b \nu_L}{\mu_A}\left(1-\frac{1}{N}\right)+\nu_L+\mu_L\right)-\lambda & \frac{r b}{N} & -r b \frac{\nu_L}{\mu_A}\left(1-\frac{1}{{N}}\right)^2\left(K_{v_{\max }}-K_{v_{\min }}\right) \\
\nu_L & -\mu_A-\lambda & 0 \\
0 & 0 &( k(- r_c+r_d)+\gamma)-\lambda
\end{array}\right].
$$

The determinant of this matrix can be calculated by expanding along the third row, i.e., 

$$
\operatorname{det}(J-\lambda I)=[( k(- r_c+r_d)+\gamma)-\lambda]\operatorname{det}\left[\begin{array}{cc}
-\left(\frac{r b \nu_L}{\mu_A}\left(1-\frac{1}{{N}}\right)+\nu_L+\mu_L\right)-\lambda & \frac{r b}{\mathcal{N}} \\
\nu_L & -\mu_A-\lambda
\end{array}\right]
$$

$$
\begin{aligned}
& =[( k(- r_c+r_d)+\gamma)-\lambda] \\
& {\left[\lambda^2+\left(\left(\frac{r b \nu_L}{\mu_A}\left(1-\frac{1}{{N}}\right)+\nu_L+\mu_L\right)+\mu_A\right) \lambda+\left(\left(\frac{r b v_L}{\mu_A}\left(1-\frac{1}{{N}}\right)+\nu_L+\mu_L\right) \mu_A-\frac{r b \nu_L}{{N}}\right)\right]}.
\end{aligned}
$$

The first eigenvalue \(\lambda_1=k(- r_c+r_d)+\gamma\) and the other two eigenvalues are given by $$
\begin{aligned}
& \lambda_{2,3}=\frac{-\left(\left(\frac{r b \nu_L}{\mu_A}\left(1-\frac{1}{{N}}\right)+\nu_L+\mu_L\right)+\mu_A\right)}{2} \pm \\
& \frac{\sqrt{\left(\left(\frac{r b \nu_L}{\mu_A}\left(1-\frac{1}{{N}}\right)+\nu_L+\mu_L\right)+\mu_A\right)^2-4\left(\left(\frac{r b \nu_L}{\mu_A}\left(1-\frac{1}{{N}}\right)+\nu_L+\mu_L\right) \mu_A-\frac{r b \nu_L}{{N}}\right)}}{2}.
\end{aligned}
$$

Since $k>0$, the eigenvalue \(\lambda_1\) will be negative if and only if $k\left(-r_c+r_d\right)+\gamma<0 \Leftrightarrow r_c-r_d>\frac{\gamma}{k}$.

The eigenvalues \(\lambda_{2,3}\) will have negative real parts if and only if

$$
\begin{aligned}
& \left(\frac{r b \nu_L}{\mu_A}\left(1-\frac{1}{{N}}\right)+\nu_L+\mu_L\right) \mu_A-\frac{r b \nu_L}{{N}}>0 \\
& \Leftrightarrow \quad\left(1-\frac{2}{{N}}\right)>-\frac{\left(\nu_L+\mu_L\right) \mu_A}{r b \nu_L} = \frac{1}{{N}} \\
\end{aligned}
$$

which is true if $N>1$. Since $N>1$ and $r_c-r_d>\frac{\gamma}{k}$ are also existence conditions for ${\tilde{E}_{03}}$, hence, ${\tilde{E}_{03}}$ is LAS, confirming the stability condition stated in item 3. 

\item \underline{$E_{04}$:} The Jacobian matrix at $E_{04}$ is given by

$$
J=\left[\begin{array}{ccc}
-\left(\frac{r b \nu_L}{\mu_A}\left(1-\frac{1}{{N}}\right)+\nu_L+\mu_L\right) & \frac{r b}{{N}} & -r b \frac{\nu_L}{\mu_A}\left(1-\frac{1}{{N}}\right)^2\left(K_{v_{\max }}-K_{v_{\min }}\right) \\
\nu_L & -\mu_A & 0 \\
0 & 0 & k(r_c - r_d) -\gamma
\end{array}\right]
$$

and the characteristic equation is given by \(\det(J - \lambda I) = 0\) where

$$
J-\lambda=\left[\begin{array}{ccc}
-\left(\frac{r b \nu_L}{\mu_A}\left(1-\frac{1}{{N}}\right)+\nu_L+\mu_L\right)-\lambda & \frac{r b}{{N}} & -r b \frac{\nu_L}{\mu_A}\left(1-\frac{1}{{N}}\right)^2\left(K_{v_{\text {max }}}-K_{v_{\text {min }}}\right) \\
\nu_L & -\mu_A-\lambda & 0 \\
0 & 0 & ((k(r_c - r_d) -\gamma)-\lambda
\end{array}\right].
$$

The determinant of this matrix can be calculated by expanding along the third row, i.e.,

$$
\operatorname{det}(J-\lambda I)=[(k(r_c - r_d)-\gamma)-\lambda]  \operatorname{det}\left[\begin{array}{cc}
-\left(\frac{r b \nu_L}{\mu_A}\left(1-\frac{1}{{N}}\right)+\nu_L+\mu_L\right)-\lambda & \frac{r b}{{N}} \\
\nu_L & -\mu_A-\lambda
\end{array}\right]
$$

$$
\begin{aligned}
& =[(k(r_c - r_d) -\gamma)-\lambda] \\
& {\left[\lambda^2+\left(\left(\frac{r b \nu_L}{\mu_A}\left(1-\frac{1}{{N}}\right)+\nu_L+\mu_L\right)+\mu_A\right) \lambda+\left(\left(\frac{r b \nu_L}{\mu_A}\left(1-\frac{1}{{N}}\right)+\nu_L+\mu_L\right) \mu_A-\frac{r b \nu_L}{{N}}\right)\right]}.
\end{aligned}
$$

The first eigenvalue \(\lambda_1=k\left(r_c-r_d\right)-\gamma\) and the other two eigenvalues are given by $$
\begin{aligned}
& \lambda_{2,3}=\frac{-\left(\left(\frac{r b \nu_L}{\mu_A}\left(1-\frac{1}{{N}}\right)+\nu_L+\mu_L\right)+\mu_A\right)}{2} \pm \\
& \frac{\sqrt{\left(\left(\frac{r b \nu_L}{\mu_A}\left(1-\frac{1}{{N}}\right)+\nu_L+\mu_L\right)+\mu_A\right)^2-4\left(\left(\frac{r b \nu_L}{\mu_A}\left(1-\frac{1}{{N}}\right)+\nu_L+\mu_L\right) \mu_A-\frac{r b \nu_L}{{N}}\right)}}{2}
\end{aligned}
$$

Since $k>0$, the eigenvalue \(\lambda_1\) will be negative if and only if $k\left(r_c-r_d\right)-\gamma<0 \Leftrightarrow r_c-r_d<\frac{\gamma}{k}$ and the other two eigenvalues \(\lambda_{2,3}\) will have negative real parts if 

$$
\begin{aligned}
& \left(\frac{r b \nu_L}{\mu_A}\left(1-\frac{1}{{N}}\right)+\nu_L+\mu_L\right) \mu_A-\frac{r b \nu_L}{{N}}>0 
\end{aligned}
$$

which is the same condition obtained for $E_{03}$. Since $N>1$ is also an existence condition for $E_{04}$, the analysis confirms that $E_{04}$ is LAS, as outlined in item 4. $E_{04}$ becomes unstable if $r_c-r_d>\frac{\gamma}{k}$.

\end{itemize}
\end{proof}

The results of the local stability analysis for the model with public health intervention are summarized in Table~\ref{tab:LAS_summary_gamma}. As with the model assuming constant payoffs (see main text), we conducted numerical simulations to illustrate these analytical findings, obtaining qualitatively similar outcomes. Figure~\ref{fig:model2} presents a colormap together with four representative trajectories of the aquatic and adult mosquito populations \((L_v, A_v)\) and the household control behavior \(w(t)\), each converging to the steady states derived in Proposition~\ref{prop: Model with PH Intervention}.

\begin{table}[ht]
\centering
\renewcommand{\arraystretch}{1.7}
{\color{black}
\begin{tabular}{|c|c|c|}
\hline
\textbf{Parameter Regime} & \boldmath{$N < 1$} & \boldmath{$N > 1$} \\
\hline
\boldmath{$r_c - r_d > \dfrac{\gamma}{k}$} 
& \begin{tabular}[c]{@{}l@{}} 
 $\tilde{E}_{01}$ is LAS. \\
 $E_{02}$ is unstable. \\
$\tilde{E}_{03}$ and $E_{04}$ do not exist.
\end{tabular}
& \begin{tabular}[c]{@{}l@{}} 
 $\tilde{E}_{03}$ is LAS. \\
 $\tilde{E}_{01}, E_{02}, E_{04}$ are unstable.
\end{tabular} \\
\hline
\boldmath{$r_c - r_d < \dfrac{\gamma}{k}$} 
& \begin{tabular}[c]{@{}l@{}} 
 $E_{02}$ is LAS. \\
 $\tilde{E}_{01}, \tilde{E}_{03}, E_{04}$ do not exist.
\end{tabular}
& \begin{tabular}[c]{@{}l@{}} 
 $E_{04}$ is LAS. \\
 $E_{02}$ is unstable. \\
 $\tilde{E}_{01}$ and $\tilde{E}_{03}$ do not exist.
\end{tabular} \\
\hline
\end{tabular}
}
\vspace{0.3em}
\caption{Local asymptotic stability (LAS) of steady states for the model with public health influence $\gamma$, under varying regimes of the basic offspring number $N$ and the threshold-adjusted behavioral payoff difference $r_c - r_d$.}
\label{tab:LAS_summary_gamma}
\end{table}

\begin{figure}[H] 
  \centering
  \includegraphics[width=\textwidth]{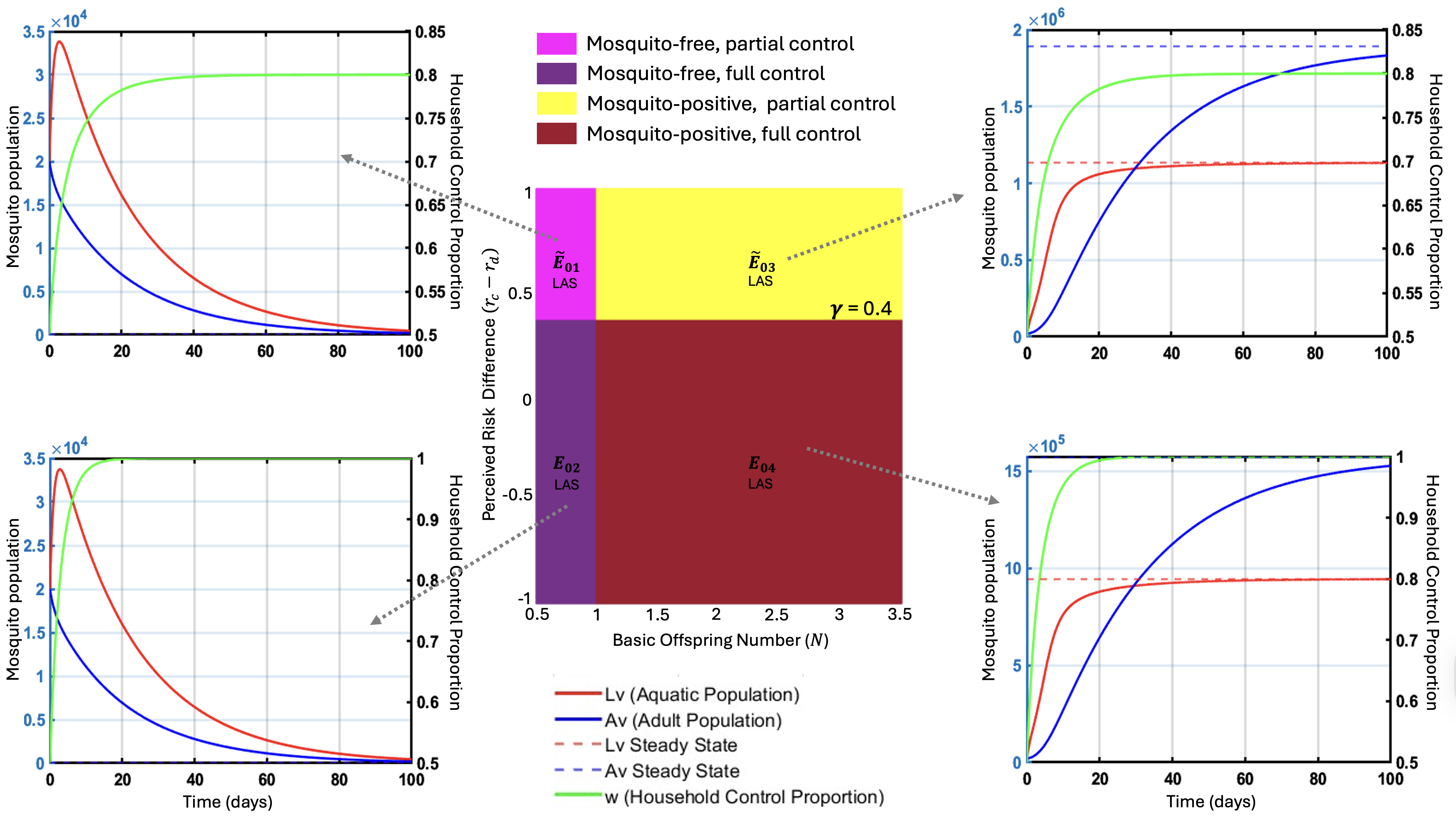}  
  \caption{Parameter regions of stability and sample trajectories for the simplified model with constant public health intervention ($\gamma > 0$). The central colormap shows the four regions in which the steady states are locally asymptotically stable (LAS). For each region, representative system trajectories are displayed, illustrating convergence to the corresponding steady state (dashed line).
  Here we chose $\gamma = 0.4$. Parameter values taken from \cite{dumont_mosquito_control}: $r = 0.5$,  $\nu_L   = 0.067\,\text{days}^{-1}$, $\mu_L = 0.62\,\text{days}^{-1}$, and $\mu_A =  0.04\,\text{days}^{-1}$. Other parameters set to plausible values: $K_{v_{\max}} = 2\times 10^{6}$, $K_{v_{\min}} = 1\times 10^{6}$, $t_{\text{span}} = [0,100]$ days, $k = 0.8\,\text{days}^{-1}$, and $b$ ranging from $1$ to $15$ to generate different $N$ values. Initial conditions: $L_0 = 20000$, $A_0 = 20000$, and $w_0 = 0.5$.}  
  \label{fig:model2}  
\end{figure}

\end{document}